\documentclass[reqno]{amsart}
\usepackage{amssymb}
\usepackage[mathcal]{euscript}
\usepackage[cmtip,all]{xy}
\usepackage{amsmath}\usepackage{amsfonts}
\usepackage{graphicx} 
\usepackage{verbatim} 
\usepackage{tikz}
\usepackage{hyperref} 
\usepackage[left=20mm, right=20mm, top=20mm, bottom=20mm]{geometry} 
\usepackage{color} 
\newcommand{\longhookrightarrow}{\lhook\joinrel\longrightarrow}
\newcommand{\bi}{\textnormal{\textbf{i}}} 
  
\newcommand{\cZ}{\mathcal{Z}} \newcommand{\FF}{\mathbb{F}} \newcommand{\KK}{\mathbb{K}}
\newcommand{\ZZ}{\mathbb{Z}} \newcommand{\NN}{\mathbb{N}} 
\newcommand{\cT}{\mathcal{T}} \newcommand{\cO}{\mathcal{O}} \newcommand{\cV}{\mathcal{V}}
\newcommand{\cW}{\mathcal{W}} \newcommand{\cH}{\mathcal{H}} \newcommand{\cS}{\mathcal{S}} 
\newcommand{\med}{\;|\;} 
 \DeclareMathOperator{\Aut}{\mathrm{Aut}}
 \DeclareMathOperator{\Stab}{\mathrm{Stab}}
 
\DeclareMathOperator{\End}{\mathrm{End}} 
\DeclareMathOperator{\Str}{\mathrm{Str}} \DeclareMathOperator{\bAut}{\mathbf{Aut}}
\DeclareMathOperator{\bPGL}{\mathbf{PGL}} 
\DeclareMathOperator{\bmu}{\boldsymbol{\mu}} 
\DeclareMathOperator{\bGL}{\mathbf{GL}}
\DeclareMathOperator{\Sim}{\mathrm{Sim}} \DeclareMathOperator{\bSim}{\mathbf{Sim}}
\DeclareMathOperator{\Iso}{\mathrm{Iso}} \DeclareMathOperator{\bIso}{\mathbf{Iso}}
\DeclareMathOperator{\Ort}{\mathrm{O}} \DeclareMathOperator{\bOrt}{\mathbf{O}} 
\DeclareMathOperator{\GO}{\mathrm{GO}} \DeclareMathOperator{\bGO}{\mathbf{GO}}
\DeclareMathOperator{\Sp}{\mathrm{Sp}} \DeclareMathOperator{\bSp}{\mathbf{Sp}}
\DeclareMathOperator{\GSp}{\mathrm{GSp}} \DeclareMathOperator{\bGSp}{\mathbf{GSp}}
\DeclareMathOperator{\Uni}{\mathrm{U}} \DeclareMathOperator{\bUni}{\mathbf{U}}
\DeclareMathOperator{\GU}{\mathrm{GU}} \DeclareMathOperator{\bGU}{\mathbf{GU}}
\newcommand{\cR}{\mathcal{R}}
\DeclareMathOperator{\chr}{\mathrm{char}\,} 
\DeclareMathOperator{\lspan}{\mathrm{span}}

\newcommand{\GL}{\mathrm{GL}} 
  \newcommand{\cM}{\mathcal{M}}
\newcommand{\cA}{\mathcal{A}} \newcommand{\cL}{\mathcal{L}}
\newcommand{\inv}{{}^-} \newcommand{\kan}{\mathfrak{K}}
 \DeclareMathOperator{\id}{\mathrm{id}}

 \DeclareMathOperator{\rank}{\textnormal{rank}}

\newcommand{\bG}{\mathbf{G}}

\newcommand{\Tr}{\mathsf{t}} 
\newcommand{\s}{\mathsf{s}} 
\newcommand{\ex}{\mathrm{ex}} 
\DeclareMathOperator{\tr}{\textnormal{t}} 
\DeclareMathOperator{\diag}{\mathrm{diag}} 

\DeclareMathOperator{\GJP}{\mathbf{GJP}} 
\DeclareMathOperator{\Ad}{\mathrm{Ad}}
\newcommand{\VI}{\cV^{\textnormal{(I)}}}
\newcommand{\VII}{\cV^{\textnormal{(II)}}}
\newcommand{\VIII}{\cV^{\textnormal{(III)}}}
\newcommand{\frg}{{\mathfrak{g}}}
\newcommand{\cE}{\mathcal{E}}
\newtheorem{theorem}{Theorem}
\newtheorem{proposition}[theorem]{Proposition}
\newtheorem{lemma}[theorem]{Lemma}
\newtheorem{corollary}[theorem]{Corollary}
\theoremstyle{definition}
\newtheorem{df}[theorem]{Definition}
\newtheorem{example}[theorem]{Example}
\newtheorem{examples}[theorem]{Examples}
\newtheorem{notation}[theorem]{Notation}
\newtheorem{remark}[theorem]{Remark}

\numberwithin{equation}{section} 
\numberwithin{theorem}{section} 
	
\begin{document}
		
\title[Aut. group schemes of Kantor pairs of assoc. central simple struct. algebras]
	{Automorphism group schemes of Kantor pairs of \\ associative central simple structurable algebras}
		
\author[D. Aranda-Orna]{Diego Aranda-Orna}
\address{Departamento de Matem\'{a}ticas,	Universidad de Oviedo, 33007 Oviedo, Spain}
\email{diego.aranda.orna@gmail.com}
		
\author[A.S. C\'ordova-Mart\'inez]{Alejandra S. C\'ordova-Mart\'inez}
\address{Departamento de Matem\'{a}tica Aplicada, Universidad de Málaga, 29071 Malaga, Spain}
\email{sarina.cordova@gmail.com}
		
\author[A. Daza-Garc\'ia]{Alberto Daza-Garc\'ia}
\address{Departamento de Ciencias Integradas, Universidad de Huelva, 21007 Huelva, Spain}
\email{alberto.daza@dci.uhu.es}

\thanks{
The three authors have been supported by grant PID2021-123461NB-C21, funded by MCIN/AEI/10.13039/501100011033 and by ``ERDF A way of making Europe''.}
\date{}
		
\begin{abstract}
In this work, a description of the automorphism group schemes is given for the Kantor pairs (and for some Kantor triple systems related to them) associated to associative central simple structurable algebras in the three ``split'' cases, where the base field $\FF$ is algebraically closed of characteristic different from $2$. 

For two of the three cases (the ones with orthogonal and symplectic involutions), we relate the decomposition (as a central product) of the automorphism group scheme to a decomposition (as a tensor product) of the corresponding Kantor pair regarded as a metric generalized Jordan pair.

\end{abstract}
		
\maketitle
		
		
\section{Introduction}

In 1978, Allison defined a new class of algebras with involution called structurable algebras \cite{AliStr}. These algebras were defined over a field $\FF$ of characteristic different from $2$ and $3$, but we will consider them in the more general setting of fields of characteristic different from $2$ which appears in \cite{AF93}.
		
\begin{df}
A \emph{structurable algebra} $(\cA,\inv)$ is a unital algebra with involution such that, for all $a,b,c,d\in\cA$, the following identities hold
\begin{align}
& [V_{a,b},V_{c,d}]=V_{V_{a,b}(c),d}-V_{c,V_{b,a}(d)}, \\
& (a-\bar{a},b,c)=(b,\bar{a}-a,c),
\end{align}
where we consider the $V$-operators defined for $a,b,c\in \cA$ by
\begin{equation}
 V_{a,b}(c) := (a\bar b) c + (c\bar b) a - (c\bar a) b,
\end{equation}
and where $(a,b,c):=(ab)c-a(bc)$. The $U$-operators are defined by $U_a(b) := \{a,b,a\}$.
\end{df}

Structurable algebras are a generalization of Jordan algebras (Jordan algebras with the identity involution are structurable algebras). Recall that the TKK construction ($3$-graded Lie algebra) from a Jordan algebra was generalized in \cite{AliModels} to a TKK construction (a $5$-graded Lie algebra, also known as the Kantor construction) over fields of characteristic different from $2$ and $3$ from a structurable algebra $(\cA,\inv)$, denoted by $\mathfrak{K}(\cA,\inv)$, and over fields of characteristic $3$ in \cite{CE25}.

A classification of structurable algebras over a field $\FF$ of characteristic $0$ was completed by Allison with a missing case in \cite[Theorem 25]{AliStr}, and was corrected and extended by Smirnov to the case that $\chr \FF \neq 2,3,5$, as follows:
\begin{theorem}[{\cite[Theorem 3.8]{S90}}]
Let $(\cA,\inv)$ be a finite-dimensional central simple structurable algebra over a field $\FF$ with $\chr \FF \neq 2,3,5$. Then $(\cA,\inv)$ is isomorphic to one of the following (non-disjoint) classes of algebras:
\begin{enumerate}
\item A central simple associative algebra with involution.
\item A central simple Jordan algebra with the identity involution.
\item A structurable algebra related to an hermitian form.
\item A form of an algebra constructed from an admissible triple (these have skew-dimension $1$).
\item A form of a tensor product of two composition algebras.
\item A Smirnov algebra.
\end{enumerate}
\end{theorem}
		
Recall that the group of automorphisms of a structurable algebra $(\cA,\inv)$, denoted $\Aut(\cA,\inv)$, consists of the automorphisms of $\cA$ which commute with the involution. We say that $\varphi\in\GL(\cA)$ is an \emph{autotopy} if there is another isomorphism $\widehat{\varphi}\in\GL(\cA)$ such that
\begin{equation}
\varphi\left(V_{x,y}(z)\right) = V_{\varphi(x),\widehat{\varphi}(y)}\varphi(z).
\end{equation}
The group of autotopies is denoted by $\Str(\cA,\inv)$, and called the \emph{structure group} of $(\cA,\inv)$.
If $\chr\FF \neq  2,3$, there are inclusions
$$	\Aut(\cA,\inv) \longhookrightarrow \Str(\cA,\inv) \longhookrightarrow \Aut(\mathfrak{K}(\cA,\inv)), $$
the second inclusion, as shown in \cite[Proposition~12.3]{AH81}, induces an isomorphism between $\Str(\cA,\inv)$ and the group of automorphisms of the $\ZZ$-graded Lie algebra $\mathfrak{K}(\cA,\inv)$ (i.e., the automorphisms preserving the $\ZZ$-grading). That isomorphism still holds for automorphism group schemes \cite[Lemma 5.5]{Rig22}.
		
Autotopies can also be regarded as automorphisms of the Kantor pair associated to $(\cA, \inv)$. Given a structurable algebra $(\cA, \inv)$, recall that its associated Kantor triple system is given by $\cT_\cA := \cA$ with the triple product $\cT_\cA \times \cT_\cA \times \cT_\cA \to \cT_\cA$ given by
\begin{equation}
\{x, y, z\} := V_{x,y}(z) = (x\bar y) z + (z\bar y) x - (z\bar x) y,
\end{equation}
for all $x, y, z \in \cT_\cA$, and its associated Kantor pair is given by $\cV_\cA = (\cV_\cA^+, \cV_\cA^-) := (\cA, \cA)$ with triple products $\cV_\cA^\sigma \times \cV_\cA^{-\sigma} \times \cV_\cA^\sigma \to \cV_\cA^\sigma$, $\{x, y, z\}^\sigma := \{x, y, z\}$ for $x, z \in\cV_\cA^\sigma$, $y \in\cV_\cA^{-\sigma}$ and $\sigma\in\{+, -\}$. We may omit the superindex $\sigma$ when it is clear by the context. We may also write $x^\sigma$ for an element $x\in\cA$ to emphasize that it belongs to $\cV_\cA^\sigma$.

An element $x\in \cA$ is called \textit{conjugate invertible} if $V_{x,y} = \id_\cA = V_{y,x}$ for some $y\in \cA$ (see \cite[\S~6]{AH81}). In that case, $y$ is called the \textit{conjugate inverse} of $x$, and is denoted by $\hat{x}$.

\medskip

Let $(\cA, \inv)$ be an associative algebra with involution. If $(\cA, \inv)$ is central simple as an algebra with involution and the base field $\FF$ is algebraically closed (with $\chr\FF\neq 2$), then there are three possible cases up to isomorphism \cite[\S~2]{BOI98}. The first case is the \emph{orthogonal case}, given by $\cA = \cM_n(\FF)$ with the transposition involution $\bar x = x^\Tr$. The second case is the \emph{symplectic case}, given again by $\cA = \cM_n(\FF)$, but with $n=2m$, and the standard symplectic involution defined by $\bar x = x^\s := \Omega x^\Tr \Omega^{-1}$, for $x\in \cM_n(\FF)$, where
\[\Omega = \left( \begin{array}{c|c} 0 & \mathrm{I}_m \\ \hline - \mathrm{I}_m & 0\end{array} \right).\]
Notice that for
\[x = \left( \begin{array}{c|c} A & B \\
\hline C & D \end{array} \right) \in \cM_{n} (\FF),\quad
\text{we have}
\quad x^\s =  \left( \begin{array}{c|c} D^t & -B^t \\
\hline -C^t & A^t \end{array} \right),\]
where $A,B,C,D \in \cM_{m} (\FF)$.
The third and last case is the \emph{unitary case}, given by $\cA=\cM_n(\FF)\oplus \cM_n(\FF)^{op}$ with the exchange involution given by $\ex(x,y) := (y,x)$.

\bigskip

This work is organized in three more sections. In Section~\ref{s:section0} we recall basic results and notations about the automorphism group schemes of some associative algebras. In Section~\ref{s:section1} we prove the main theorems. In Theorem~\ref{automorphismschemes} we compute the automorphism group schemes of the Kantor pair and triple system related to an associative algebra with the orthogonal or the symplectic involution, and in Theorem~\ref{automorphismschemes.unitar} we compute the automorphism group schemes of the Kantor pair and triple system related to an associative algebra with the unitary involution. In section \ref{s:section2} we decompose the Kantor pairs related to an associative algebra with the orthogonal and symplectic involutions into the tensor products of some simpler metric generalized Jordan pairs (in the former case, the decomposition was already known).

From now on, $\FF$ will denote an algebraically closed field of characteristic different from 2.

\section{Preliminaries}\label{s:section0}

We will recall now some notation from \cite[\S~12.B]{BOI98}. Let $(\cA,\inv)$ be a central simple associative algebra with involution. Recall that a \emph{similitude} of $(\cA,\inv)$ is an element $x\in \cA$ such that $\bar{x}x\in \FF^\times$. The scalar $\mu(x) := \bar{x}x=x\bar{x}$ is called the \emph{multiplier of $x$}. An element $x\in \cA$ such that $\mu(x) = 1$ is an \emph{isometry} of $(\cA,\inv)$. The groups of similitudes and isometries of $(\cA,\inv)$ are respectively denoted by $\Sim(\cA,\inv)$ and $\Iso(\cA,\inv)$. Depending on if the involution is of orthogonal, symplectic, or unitary type, the groups of similitudes of $(\cA,\inv)$ are denoted by 
$\GO(\cA,\inv)$, $\GSp(\cA,\inv)$, or $\GU(\cA,\inv)$, and the groups of isometries by
$\Ort(\cA,\inv)$, $\Sp(\cA,\inv)$, or $\Uni(\cA,\inv)$. The corresponding group schemes are respectively denoted by $\bSim(\cA,\inv)$, $\bIso(\cA,\inv)$, $\bGO(\cA,\inv)$, $\bGSp(\cA,\inv)$, $\bGU(\cA,\inv)$, $\bOrt(\cA,\inv)$, $\bSp(\cA,\inv)$, and $\bUni(\cA,\inv)$.
Other important affine group schemes are: the multiplicative group scheme $\bG_m := \bGL_1$, whose group of $\cR$-points (where $\cR$ is an associative commutative unital $\FF$-algebra) is the multiplicative group $\cR^{\times}$; the group scheme $\bmu_n$, which is the subgroup scheme of $\bG_m$ whose group of $\cR$-points consists of the elements whose order divides $n$; the constant group scheme of a group $G$, denoted also as $G$, which is represented by the Hopf algebra $\prod_{g\in G}\FF e_g$ where the elements $e_g$ are idempotents, and the coproduct is given by $\Delta(e_h)=\sum_{g_1g_2=h}e_{g_1}\otimes e_{g_2}$ (see \cite[\S~20.A]{BOI98}).

When dealing with affine group schemes, we will denote scalar extensions as $V_\cR := V \otimes_{\FF} \cR$, where $V$ is a vector space over the field $\FF$ and $\cR$ is a unital associative commutative algebra.

\medskip

Let us recall the definition of central product of groups (see \cite[Chapter 2]{G80}). Given two groups $G_1$ and $G_2$, two isomorphisms $\gamma_i: H \rightarrow H_i$ such that $H_i \leq Z(G_i)$, for $i=1,2$, let $N_{\gamma}:= \lbrace (\gamma_1(h),\gamma_2(h)^{-1}) \vert h\in H \rbrace = \lbrace (h, \gamma(h)^{-1}) \vert h\in H_1 \rbrace$ for $\gamma:= \gamma_2 \circ \gamma_1^{-1}: H_1 \rightarrow H_2$. Then, the group $(G_1 \times G_2) / N_{\gamma}$ is called \textit{the (outer) central product of $G_1$ and $G_2$ relative to $\gamma$}. If the isomorphisms $\gamma_i$ are clear by context, we will denote it by $G_1 \otimes_H G_2$.

\begin{remark} \label{rem.schemes}	

Let $\cA = \cM_n(\FF)$. Recall from \cite[\S~23]{BOI98} that we have the adjoint morphism $\text{Ad} \colon \bGL_n \to \bAut(\cA)$ which is an epimorphism with kernel $\bG_m := \bGL_1$. This implies $\bPGL_n := \bGL_n / \bG_m \simeq \bAut(\cA)$; furthermore, due to the sheaf property of quotients \cite[\S~15.5]{W79}, given an element $\varphi\in\bAut(\cA)(\cR) = \Aut_\cR(\cA_\cR)$, there is a faithfully flat extension $\iota\colon \cR\to \cS$ such that $\bAut(\cA)(\iota)(\varphi)$ is a conjugation by an element of $\bGL_n(\FF)(\cS) = \GL_n(\cS)$. It is easy to see that if $\inv$ is an involution on $\cA$, then for any element $\varphi\in\bAut(\cA, \inv)(\cR) = \Aut_\cR(\cA_\cR, \inv)$ there is a faithfully flat extension $\iota\colon \cR\to \cS$ such that $\bAut(\cA, \inv)(\iota)(\varphi)$ is a conjugation by an element of $\bSim(\cA, \inv)(\cS) = \Sim_\cR(\cA_\cS, \inv)$. It follows that $\bAut(\cA, \inv) \simeq \bSim(\cA, \inv) / \bG_m$.

\end{remark}
\begin{remark}\label{rem.schemeUnitaryCase}
Let $\cA=\cM_n(\FF)\oplus \cM_n(\FF)^{op}$ with the exchange involution $\ex(x,y):=(y,x)$. Denote the center of $\cA$ by $\KK:=\cZ(\cA)=\FF (\mathrm{I}_n,0)\oplus \FF(0,\mathrm{I}_n)$. As shown in \cite[\S~23.A]{BOI98}, $\bAut(\KK)\cong \bmu_2$ the constant group scheme of cyclic group of order 2 generated by $\ex$. Moreover, it is shown that  there is an epimorphism
\begin{equation*}
\pi \colon \bAut(\cA,\ex) \longrightarrow \bAut(\KK) \cong \bmu_2
\end{equation*}
given by the restriction. Its kernel is $\bAut_{\KK}(\cA,\ex)$, the group scheme of automorphisms which fix the center. There is an isomorphism
\begin{equation}\label{eq:unitaryIota}
\iota \colon \bAut(\cM_n(\FF)) \longrightarrow \bAut_{\KK}(\cA,\ex)
\end{equation}
given by
$$ \varphi \longmapsto (\varphi,\varphi^{op}). $$
Thus, there is a short exact sequence
\begin{equation}\label{eq.exactSequenceUnitary}
1 \longrightarrow \bAut(\cM_{n}(\FF))\overset{\iota}{\longrightarrow}
\bAut(\cA,\ex)\overset{\pi}{\longrightarrow} \bmu_2 \longrightarrow 1.
\end{equation}
The sequence splits via the morphism
$$ \bmu_2 \longrightarrow \bAut(\cA,\ex), $$
which sends $r\in \bmu_2(\cR)$ to the morphism $(x,y)\mapsto (y^\Tr,x^\Tr)\otimes\frac{1}{2}(1-r)+(x,y)\otimes\frac{1}{2}(1+r)$.
Thus,
\begin{equation}\label{eq:IsomUnitary}
\bAut(\cA,\ex) \cong \bAut_{\KK}(\cA,\ex)\rtimes  \bmu_2.
\end{equation}
\end{remark}

\section{Kantor pairs of associative algebras with involution} \label{s:section1}
		
In this section, $(\cA, \inv)$ will denote a (finite-dimensional) central simple associative algebra with involution. (Note that $\cA$ may not be central simple as an algebra without involution). The associated Kantor pair and Kantor triple system will be denoted by $\cV = \cV_\cA$ and $\cT = \cT_\cA$, respectively.

\begin{remark}\label{rm:inclusions}
	Given a structurable algebra $(\cA,\inv)$ there are well-known inclusions:
	
	\[\begin{matrix}
		\bAut(\cA,\inv)&\longrightarrow&\bAut(\cT_{\cA})&\longrightarrow&\bAut(\cV_{\cA})\\
		\psi&\longmapsto& \psi&\longmapsto& (\psi,\psi)
	\end{matrix}\]
	which we will denote by $\iota$.
\end{remark}
		
\begin{notation} \;
\begin{enumerate}
\item[$1)$]
Denote by $L_a$ and $R_a$, respectively, the left and right multiplication by $a\in\cA$.
For $a\in\cA^\times$, we claim that 
\begin{equation}
\widetilde{L}_a := (L_a, L_{\bar{a}}^{-1}) \in \Aut \cV.
\end{equation}
Indeed, we have
\begin{align*}
& L_a(x) \overline{L_{\bar a}^{-1}(y)} L_a(z) = (ax) (\overline{\bar{a}^{-1}y}) (az)
= ax \bar{y}(a^{-1}a)z = L_a(x \bar{y} z), \\
& L_a(z) \overline{L_a(x)} L_{\bar a}^{-1}(y) = (az) (\overline{ax}) (\bar{a}^{-1}y)
= az \bar{x} (\bar{a} \bar{a}^{-1}) y = L_a(z \bar{x} y), \\
\end{align*}
and replacing $a$ by $\bar{a}^{-1}$ we also get
\begin{align*}
L_{\bar a}^{-1}(x) \overline{L_a(y)} L_{\bar a}^{-1}(z) = L_{\bar a}^{-1}(x \bar{y} z),
\qquad
L_{\bar a}^{-1}(z) \overline{L_{\bar a}^{-1}(x)} L_a(y) = L_{\bar a}^{-1}(z \bar{x} y).
\end{align*}
Then the claim follows easily. Note that for each $a \in \Sim(\cA, \inv)$, we have
\begin{equation}
\widetilde{L}_a = \widehat{L}_a := (L_a, \mu(a)^{-1}L_a).
\end{equation}
\item[$2)$]
Given $a\in \Iso(\cA,\inv)$, we claim that $L_a, R_a \in\Aut \cT$. Indeed, the claim follows from
\begin{align*}
& L_a(x) \overline{L_a(y)} L_a(z) = (ax) (\overline{ay}) (az)
= ax \bar{y}(\bar{a}a)z = ax \bar{y}z = L_a(x \bar{y} z), \\
& R_a(x) \overline{R_a(y)} R_a(z) = (xa) (\overline{ya}) (za)
= x(a\bar{a})\bar{y}za = x \bar{y} za = R_a(x \bar{y} z).
\end{align*}
For $a\in \Sim(\cA,\inv)$, it is easy to see that
\begin{equation}
\widehat{R}_a := (R_a, \mu(a)^{-1} R_a) \in \Aut \cV.
\end{equation}
Note that $\widehat{R}_a = R_a$ if $a \in \Iso(\cA,\inv)$.
\item[$3)$]
Note that for each $\lambda\in\FF^\times$, we have $c_\lambda = (c_\lambda^+, c_\lambda^-) := \widetilde{L}_{\lambda 1} \in \Aut \cV$, where $c_\lambda^\sigma(x) := \lambda^{\sigma 1}x$ for $\sigma = \pm$. (Note that $\{c_\lambda \med \lambda\in\FF^\times \}$ is the $1$-torus producing the $\ZZ$-grading associated to the Kantor construction from $\cV$.)
\end{enumerate}
\end{notation}

\bigskip

It is well-known \cite[\S~10, Example~(1)]{AH81} that if $(\cA, \inv)$ is a finite-dimensional unital associative $\FF$-algebra with involution over a field $\FF$ of characteristic different from $2$ and $3$, then an element $u\in\cA$ is conjugate invertible if and only if $u$ is invertible; moreover, in this case we have that $\hat u = {\bar u}^{-1}$. We will now prove that, in some cases, we can extend that result to the case with $\chr\FF = 3$, as follows:

\begin{lemma} \label{lemma.inverse}
Let $\cR$ be an associative commutative unital $\FF$-algebra. Consider an algebra with involution $(\cA, \inv)$ with $\cA := \cM_n(\FF)$. Then, an element $x\in \cA_\cR = \cM_n(\cR)$ is conjugate invertible if and only if $x$ is invertible in $\cA_\cR$. In addition, in that case we have $\hat x = \bar{x}^{-1}$.
\end{lemma}
\begin{proof}
Let $y$ be a conjugate inverse of $x$, thus $V_{x,y} = \id$. Thus for each $z \in \cA_\cR$ we have that $z = V_{x,y}(z) = x\bar{y} z + z(\bar{y} x - \bar{x} y)$, or equivalently, $(1 - x \bar{y})z = z(\bar{y} x - \bar{x} y)$. By taking $z = 1$, we deduce that $a := 1 - x \bar{y} = \bar{y} x - \bar{x} y$. Since $az = za$ for each $z \in \cA_\cR$, we have $a \in Z(\cA_\cR) = \cR 1$. Note that $a = \bar{y} x - \bar{x} y \in \cS(\cA_\cR, \inv)$. Thus $a \in \cR 1 \cap \cS(\cA_\cR, \inv) = 0$ and $x\bar{y} = 1$. Since $\cA_\cR$ is a matrix algebra, it follows that $x,y\in \cA_\cR^\times$ and $x^{-1} = \bar{y}$. The converse of the result is obvious.
\end{proof}

\smallskip

\begin{lemma}\label{l:AutTriples}
Let $n \in \NN$, $\cA := \cM_n(\FF)$ and $\inv$ either the transposition or the standard symplectic involution on $\cA$. Let $\cT$ be the Kantor triple system associated to $(\cA, \inv)$. Let $\iota$ be the inclusion given in Remark \ref{rm:inclusions}. Given an associative commutative unital $\FF$-algebra $\cR$, $\varphi\in\bAut(\cT)(\cR)$, $a\in\bGL_n(\cR) = \cA^\times_\cR$ and $b\in\bSim(\cA,\inv)(\cR)$, such that $\iota_{\cR}(\varphi)=\widetilde{L}_{a}\widehat{R}_{\overline{b}}$, then $a\in\bSim(\cA,\inv)$ and $\mu(a) = \mu(b)^{-1}$.
\end{lemma}		
\begin{proof}
	 Notice that $\iota_{\cR}(\varphi)= (\varphi^+, \varphi^-)$ satisfies $\varphi^+ = \varphi^-$.  Equality $\varphi^+(x) = \varphi^-(x)$ implies that $\widetilde{a} x = x$ for all $x\in\cA_\cR$, where $\widetilde{a} := \mu(b) \bar{a} a$. It follows that $\widetilde{a} = 1$. Thus $a \in \Sim_\cR(\cA_\cR,\inv)$ with $\mu(a) = \mu(b)^{-1}$
\end{proof}

\smallskip

\begin{theorem} \label{automorphismschemes}
Let $n \in \NN$, $\cA := \cM_n(\FF)$ and $\inv$ either the transposition or the standard symplectic involution on $\cA$. Let $\cT$ and $\cV$ be the Kantor triple system and Kantor pair associated to $(\cA, \inv)$. Then, there are isomorphisms of affine group schemes
\begin{align}
& \bAut(\cV) \simeq \bG_m(\cA) \otimes_{\bG_m} \bSim(\cA,\inv), \\
& \bAut(\cT) \simeq \bIso(\cA,\inv)\otimes_{\bmu_2}\bIso(\cA,\inv),
\end{align}
where we denote $\bG_m(\cA)(\cR) := \bG_m(\cA_\cR) = \cA^\times_\cR$ (so that $\bG_m(\cA) = \bGL_n$). \\
$\blacksquare$ In particular, for the transposition involution we have
\begin{align}
& \bAut(\cV) \simeq \bGL_n \otimes_{\bG_m} \bGO_n, \\
& \bAut(\cT) \simeq \bOrt_n \otimes_{\bmu_2} \bOrt_n,
\end{align}
and for the standard symplectic involution with $n = 2m$ we have
\begin{align}
& \bAut(\cV) \simeq \bGL_n \otimes_{\bG_m} \bGSp_n, \\
& \bAut(\cT) \simeq \bSp_n \otimes_{\bmu_2} \bSp_n.
\end{align}
\end{theorem}
\begin{proof}
Let $\cR$ be an associative commutative unital $\FF$-algebra. Note that $\cA^\times_\cR = \cM_n(\cR)^\times = \GL_n(\cR)$. Consider the group scheme morphism
\begin{equation}\label{eq:IsoAutV}
\Phi\colon \bGL_n \times \bSim(\cA, \inv) \longrightarrow \bAut(\cV)
\end{equation}

\noindent given by $\Phi_{\cR}(a,b)=\widetilde{L}_a\widehat{R}_b$. We will show that it is an epimorphism and that the kernel is $\bG_m$.

 Let $\varphi = (\varphi^+, \varphi^-)\in\Aut_\cR(\cV_\cR)$ and set $x^+ = \varphi^+(1^+)$. Since automorphisms send conjugate invertible elements to conjugate invertible elements, by Lemma~\ref{lemma.inverse} we have that $x^+\in\cA_\cR^\times$. Therefore, there exists $a = (x^+)^{-1} \in\cA_\cR^\times$ with $\widetilde{L}^+_a (x^+) = 1^+$. Thus $\psi := \widetilde{L}_a \varphi \in \Stab_{\Aut_\cR(\cV_\cR)}(1^+)$. Since automorphisms preserve conjugate inverses, and the only conjugate inverse of $1^+$ is $1^-$ (see Lemma~\ref{lemma.inverse}), it follows that $\Stab_{\Aut_\cR(\cV_\cR)}(1^+) = \Stab_{\Aut_\cR(\cV_\cR)}(1^+,1^-)$.
			
We claim that $\cH(\cA, \inv) \cup \cS(\cA, \inv)^2$ generates $\cA$ as an $\FF$-algebra, where we denote
$$ \cS(\cA, \inv)^2 := \lspan\{ s_1 s_2 \med s_i \in \cS(\cA, \inv) \}. $$
To prove the claim, first consider the case with the transposition involution. The case with $n = 1$ is trivial. If $n > 2$, for any different $i,j,k$ we have $E_{ik} = (E_{ij}-E_{ji})(E_{jk}-E_{kj}) \in\cS(\cA, \inv)^2$, also $E_{ii}\in\cH(\cA, \inv)$ for each $i$. If $n = 2$, it suffices to see that $\cS(\cA, \inv) = \FF(E_{12} - E_{21})$ where $E_{12} - E_{21} = (E_{11} - E_{22})(E_{12} + E_{21})$. Now, consider the case with the symplectic involution. If $n = 2$, then $\cA$ is the split quaternion algebra, and it is easy to see that $\cS(\cA, \inv)^2$ generates $\cA$ as an $\FF$-algebra. Now assume $n > 2$. If $i\neq k$, then $\diag(E_{ij},E_{ji})\diag(E_{jk},E_{kj}) = \diag(E_{ik},0)$, $\diag(E_{ik},0)\diag(E_{ki},0) = \diag(E_{ii},0)$ are matrices in the generated algebra, and the same arguments work with matrices in the lower-right block. If we multiply matrices of the form $\diag(E_{ik},0)$ or $\diag(0,E_{ik})$ by hermitian matrices of the antidiagonal blocks, the claim follows easily.

Since the above claim holds, by \cite[Proposition~3.1]{AC21} we get that $\bAut(\cA,\inv)$ is the stabilizer of $(1^+,1^-)$ in $\bAut(\cV)$. It follows that $\psi^+ = \psi^-$ and $\psi\in\Aut_\cR(\cA_\cR, \inv)$. From Remark~\ref{rem.schemes}, we know that there is a faithfully flat extension $\tau\colon \cR\to \cS$ and $c\in\bSim(\cA,\inv)(\cS)$ such that $\bAut(\cA,\inv)(\tau)(\psi_+)=\gamma_c$. Using Remark \ref{rm:inclusions} it follows that $\bAut(\cV)(\tau)(\psi)=\widehat{L}_c\widehat{R}_{c^{-1}}=\widetilde{L}_c\widehat{R}_{c^{-1}}$. Therefore $\bAut(\cV)(\tau)(\varphi)=\bAut(\cV)(\tau)(\widetilde{L}_{a^{-1}})\widetilde{L}_c\widehat{R}_{c^{-1}}=\Phi_{\cS}(\tau(a^{-1})c,c^{-1})$. Therefore, due to the sheaf property of quotients \cite[\S~15.5]{W79} it follows that $\Phi$ is an epimorphism.

			Take $(a,b)\in\ker\Phi_\cR$. Then $\Phi_\cR(a,b) = \widetilde{L}_a \widehat{R}_{\bar{b}}$ is the identity map, so that we get $ax\bar{b} = x$, or equivalently $ax = x(\bar{b})^{-1}$, for each $x\in\cA_\cR$. It is easy to see that this implies that $a = r 1 = b^{-1}$ for some $r\in \cR^\times$. We have shown that $\ker\Phi_\cR = \{ (r1,r^{-1}1) \med r \in \cR^\times \} \equiv \cR^\times = \bG_m(\cR)$. Hence the first isomorphism follows.

In order to prove the isomorphism for $\bAut(\cT)$, consider the morphism of group schemes
\begin{equation}
\Psi\colon \bIso(\cA,\inv)\times \bIso(\cA,\inv)\longrightarrow \bAut(\cT)
\end{equation}
given by $\Psi_{\cR}(a,b)=L_aR_{\overline{b}}$. We will first show that $\Psi$ is an epimorphism. Given $\varphi\in\bAut(\cT)(\cR)$, since $\Phi$ is an epimorphism, by the sheaf property of quotients \cite[\S~15.5]{W79} it follows that there is a faithfully flat extension $\tau\colon\cR\to\cS$, $a\in\bG_m(\cA)(\cR) = \GL_n(\cR)$ and $b\in\bSim(\cA,\inv)(\cR)$ such that $\iota_{\cS}(\bAut(\cT)(\tau)(\varphi))=\bAut(\cT)(\tau)(\iota_{\cR}(\varphi))=\widetilde{L}_{a}\widehat{R}_{\overline{b}}$. Due to Lemma \ref{l:AutTriples}, it follows that $a\in\bSim(\cA,\inv)(\cR)$ and $\mu(a)=\mu(b)^{-1}$. Denote $\mathcal{U}:=\frac{\cS[T]}{\langle T^2-\mu(b)^{-1}\rangle}$ and denote by $t$ a representative of $T$ in the quotient. Consider the faithfully flat extension $\theta\colon \cS\to \mathcal{U}$ given by the inclusion. We get that $\iota_{\mathcal{U}}(\bAut(\cT)(\theta\tau)(\varphi))=\bAut(\cV)(\theta\tau)(\iota_{\cR}(\varphi))=\widetilde{L}_{a}\widehat{R}_{\overline{b}}=\widetilde{L}_{t^{-1}a}\widehat{R}_{t\overline{b}}$. Thus, we get that $\bAut(\cT)(\theta\tau)(\varphi)=L_{t^{-1}a}R_{t\overline{b}}=\Psi_{\mathcal{U}}(t^{-1}a,t\overline{b})$. Therefore, due to the sheaf property of quotients \cite[\S~15.5]{W79} it follows that $\Psi$ is an epimorphism. Now, as with $\Phi$ we can show that $\ker\Psi=\bmu_2$, which implies the second isomorphism.
\end{proof}

\begin{lemma}\label{lemma.HermitianSkewHermitianCubed}
Let $\cA := \cM_n(\FF) \oplus \cM_n(\FF)^{op}$ with $n>1$ and $\inv$ the exchange involution of $\cA$. Then, $\cH(\cA,\inv)^3=\cS(\cA,\inv)^3=\cA$.	
\end{lemma}
\begin{proof}
Let $\sigma=\pm$. For $i\neq j$ we have the following:
\begin{align*}
& (E_{ik},\sigma E_{ik})(E_{kj},\sigma E_{kj})=(E_{ij},0), \\
& (E_{kj},\sigma E_{kj})(E_{ik},\sigma E_{ik})=(0,E_{ij}).
\end{align*}
Hence, it follows that $(E_{,j},0)$ and $(0,E_{ij})$ belong to $\cS(\cA,\inv)^2$ and to $\cH(\cA,\inv)^2$. Finally:
\begin{align*}
& (E_{ij},0)(E_{ji},\sigma E_{ji})=(E_{ii},0), \\
& (0,\sigma E_{ji})(E_{ij},\sigma E_{ij})=(0,E_{ii}).
\end{align*}
\end{proof}

\begin{lemma}\label{lemma.Stab1PreserveCenter}
Let $\FF$ be a field of characteristic $3$, $\cA := \cM_n(\FF) \oplus \cM_n(\FF)^{op}$ with $n>1$ and $\inv$ the exchange involution of $\cA$. Let $\varphi=(\varphi^+,\varphi^-)\in \Stab_{\bAut(\cV_{\cA})(\cR)}(1^+)$. Then:
\begin{itemize}
\item[(1)] $\varphi^-(\cS(\cA,\inv))=\cS(\cA,\inv)$, and
\item[(2)] $\varphi^{\sigma}(\cZ(\cA_{\cR}))=\cZ(\cA_{\cR})$ for $\sigma=\pm$.
\end{itemize}
\end{lemma}
\begin{proof}
$(1)$ An element $y\in \cA$ is skew-symmetric if and only if $\{1,y,1\}^+=0$. Since $\varphi^+(1)=1$, $y$ is skew-symmetric if and only if $\varphi^-(y)$ is skew-symmetric.
	
$(2)$ Let $y\in\cS(\cA,\inv)$. We have that:
\[\{y,1,z\}^-=yz-zy.\]
Then, in view of Lemma \ref{lemma.HermitianSkewHermitianCubed} $z\in\cZ(\cA,\inv)$ if and only if $\{\cS(\cA,\inv),1,z\}^-=0$. Using $(1)$, it follows that  $z\in\cZ(\cA,\inv)$ if and only if $\varphi^-(z)\in\cZ(\cA,\inv)$. In a simpler way, using $\{1,y,z\}^+$, we can prove $z\in\cZ(\cA,\inv)$ if and only if $\varphi^+(z)\in\cZ(\cA,\inv)$.
\end{proof}

\begin{lemma}\label{lemma.AutKantor}
Let $\chr(\FF)=3$, $\cA=\FF\oplus \FF$ and $\inv$ the exchange involution. Let $\varphi=(\varphi^+,\varphi^-)\in \GL_2(\cR)\times\GL_2(\cR) $. Recall that $n(x,y)=x\bar{y}+y\bar{x}$ is the polar form of the norm of $(\cA, \inv)$ as a Hurwitz algebra. Then $\varphi\in \bAut(\cV_{\cA})(\cR)$ if and only if $n(\varphi^+(x),\varphi^-(y))=n(x,y)$ for all $x,y\in\cA_{\cR}$.
\end{lemma}
\begin{proof}
Since $\cA_{\cR}$ is conmmutative and associative, we have that:
\[\{x,y,z\}^{\sigma}=-z(x\bar{y}+y\bar{x})=-zn(x,y).\]
From here the result follows trivially just by applying $\varphi^{\sigma}$.
\end{proof}

\begin{corollary}\label{corollary.AutDim2}
Let $\chr(\FF)=3$, $\cA=\FF\oplus \FF$ and $\inv$ the exchange involution.
Let $\varphi=(\varphi^+,\varphi^-)\in\GL_2(\cR)\times \GL_2(\cR)$.  Then, $\varphi\in \Stab_{\bAut(\cV_{\cA})(\cR)}(1^+)$ if and  only if there are $\beta\in \cR^{\times}$ and $\alpha\in \cR$ such that $\varphi^+(1)=1$, $\varphi^+(s)=\alpha 1 + \beta s$, $\varphi^-(1)=1+\frac{\alpha}{\beta}s$, $\varphi^-(s)=\frac{1}{\beta}s$, where $s=(1_\FF,-1_\FF)$.
\end{corollary}
\begin{proof}
	Since  $\varphi\in \Stab_{\bAut(\cV_{\cA})(\cR)}(1^+)$ we have  $\varphi^+(1)=1$, $\varphi^+(s)=\alpha 1 + \beta s$ for some $\alpha,\beta\in \cR$. Since $\varphi^+$ is invertible, then $\beta$ must be invertible. Now we can recover $\varphi^-$ using the polar form of the norm. Let $\lambda,\gamma \in \cR$ be such that $\varphi^-(1)=\lambda 1 +\gamma s$. Then, by Lemma \ref{lemma.AutKantor}, we have:
	
	\[2=n(1,1)=n(1,\lambda 1 +\gamma s)=2\lambda,\]
	from where we get $\lambda=1$ and
	\[0=n(s,1)=n(\alpha 1+\beta s,1+\gamma s)=2\alpha-2\beta\gamma,\]
	which implies $\gamma=\frac{\alpha}{\beta}$ so $\varphi^-(1)=1+\frac{\alpha}{\beta}s$. Similarly, we get the result for $s$. The converse follows by Lemma~\ref{lemma.AutKantor}.
\end{proof}

\begin{proposition}\label{proposition.invertToInvert}
Let $\cA := \cM_n(\FF) \oplus \cM_n(\FF)^{op}$, $\inv$ the exchange involution of $\cA$ and $\cV$  the Kantor pair associated to $(\cA, \inv)$. Let $\cR$ be an associative commutative unital algebra and $\varphi\in \bAut(\cV)(\cR)$. Then
 $\varphi^\sigma(\cA^\times_\cR) \subseteq \cA^\times_\cR$. In particular $\varphi^\sigma(1)$ is invertible.
\end{proposition}
\begin{proof} Without loss of generality, we can assume $\sigma = +$. Consider the maps
\begin{align*}
& [\cdot,\cdot,\cdot]^{\sigma}\colon \cV_\cR^{\sigma}\times \cV_\cR^{-\sigma}\times \cV_\cR^{\sigma}\to \cV_\cR^{\sigma}, \\
& [x,y,z]^{\sigma} := \{x,y,z\}^{\sigma}-\{z,y,x\}^{\sigma}=(x\overline{z}-z\overline{x})y.
\end{align*}
Then
$$ \cS(\cA_\cR,\inv)y=[\cA_\cR,y,1]\subseteq[\cA_\cR,y,\cA_\cR]\subseteq\cS(\cA_\cR,\inv)y, $$
which implies that 
\begin{equation}\label{newProd.SkewSim}
[\cA_\cR,y,\cA_\cR]=\cS(\cA_\cR,\inv)y.
\end{equation}
Besides $[\varphi^{\sigma}(x),\varphi^{-\sigma}(y),\varphi^{\sigma}(z)]^{\sigma}=\varphi^{\sigma}([x,y,z]^{\sigma})$ for all $x,z\in\cV_\cR^{\sigma}$ and $y\in \cV_\cR^{-\sigma}$, and notice that due to \eqref{newProd.SkewSim} and Lemma \ref{lemma.HermitianSkewHermitianCubed} we have
\begin{equation} \label{eq.trescorchetes}
\Big[ \cA_\cR, \big [\cA_\cR, [\cA_\cR,y,\cA_\cR]^{\sigma}, \cA_\cR \big]^{-\sigma}, \cA_\cR \Big]^{\sigma}
= \cS(\cA_\cR,\inv)^3y = \cA_\cR y.
\end{equation}
If we apply $\varphi^-$ to \eqref{eq.trescorchetes} with $y \in \cA^\times_\cR$, we get
$$ \cA_\cR = \varphi^-(\cA_\cR y) = \cA_\cR\varphi^+(y). $$
Thus, if $\varphi^{+}(y)=(A,B)$ with $A,B\in \cM_{n}(\cR)$, it follows that $A$ and $B$ are invertible and this implies that $\varphi^+(y)$ is invertible.
\end{proof}

\begin{lemma}\label{lemma.OrthogonalIdempotent}
Suppose that $\chr(\FF)=3$ and let $\cR$ be an associative commutative unital $\FF$-algebra.
Let $a,b\in \cR$ be such that $a+b=1$ and $a^2b=ab^2=0$. Then $a$ and $b$ are ortogonal idempotents.
\end{lemma}
\begin{proof}
	$(a+b)^3=a^3+b^3$. Therefore, $a^3+b^3=1$. Now, $a^2=a^2(a+b)=a^3$. Similarly, $b^2=b^3$. Thus, $a^2+b^2=1$. Now, $1=(a+b)^2=a^2+b^2+2ab=1+2ab$. Therefore $ab=0$. Finally, $a=a(a+b)=a^2$ and $b=b(a+b)=b^2$. Therefore, $a$ and $b$ are ortogonal idempotents.
\end{proof}

\begin{lemma}\label{lemma.ImageOfUe1}
	Let $\chr(\FF)=3$, $\cA=\cM_n(\FF)\oplus \cM_n(\FF)^{op}$ and $\inv$ the exchange involution. Let $r_1,r_2\in \cR$ be such that $r_1+r_2=1$ and $\lambda\in \cR^{\times}$. Then \[\cH(\cA,\inv)(r_1e_1+ r_2e_2)\subseteq U_{\lambda r_1e_1+\lambda r_2e_2}(\cA)\text{, and  }\cS(\cA,\inv)(r_1e_1+ r_2e_2)\subseteq U_{\lambda r_1e_1-\lambda r_2e_2}(\cA).\] where $e_1=(\mathrm{I}_n,0)$, $e_2=(0,\mathrm{I}_n)$ and $\sigma=\pm$.
\end{lemma}
\begin{proof}
 Take $x\in\cH(\cA,\inv)$, then
 \begin{equation*}
 	\begin{aligned}
 		U_{\lambda r_1e_1+\lambda r_2e_2}(-\lambda^{-2}x)&=2(\lambda r_1 e_1+\lambda r_2e_2)^2\overline{(-\lambda^{-2}x)}-(\lambda r_1 e_1+\lambda r_2e_2)\overline{(\lambda r_1 e_1+\lambda r_2e_2)}(-\lambda^{-2}x)\\
 		&=(r_1^2e_1+r_2^2e_2)x+(r_1r_2e_1+r_1r_2e_2)x\\
 		&=[r_1(r_1+r_2)e_1+r_2(r_1+r_2)e_2]x\\
 		&=(r_1e_1+r_2e_2)x.
\end{aligned}
\end{equation*}
Similarly, for $x\in\cS(\cA,\inv)$,
  \begin{equation*}
 	\begin{aligned}
 		U_{\lambda r_1e_1-\lambda r_2e_2}(\lambda^{-2}x)&=2(\lambda r_1 e_1-\lambda r_2e_2)^2\overline{\lambda^{-2}x}-(\lambda r_1 e_1-\lambda r_2e_2)\overline{(\lambda r_1 e_1-\lambda r_2e_2)}\lambda^{-2}x\\
 		&=(r_1^2e_1+r_2^2e_2)x+(r_1r_2e_1+r_1r_2e_2)x\\
 		&=[r_1(r_1+r_2)e_1+r_2(r_1+r_2)e_2]x\\
 		&=(r_1e_1+r_2e_2)x.
 	\end{aligned}
 \end{equation*}
\end{proof}

\begin{lemma}\label{lemma.Stab1}
Let $\cA := \cM_n(\FF) \oplus \cM_n(\FF)^{op}$ with $n>1$ and $\inv$ the exchange involution of $\cA$. Then
$$ \Stab_{\bAut(\cV_{\cA})(\cR)}(1^+)=\Stab_{\bAut(\cV_{\cA})(\cR)}(1^+,1^-). $$
\end{lemma}
\begin{proof} Let $\varphi \in \Stab_{\bAut(\cV_{\cA})(\cR)}(1^+)$. By abuse of notation we will write $1$ instead of $1^+$ and $1^-$.

$\bullet$ Case 1: $\chr(\FF)\neq 3$. Let $(a,b) := \varphi^-(1)$. Notice that $1 = \{1,1,1\}$, then $1 = \varphi^+(1) = \varphi^+\{1,1,1\} = \{1,\varphi^-(1),1\} = 2\overline{\varphi^-(1)} - \varphi^-(1) = (2b-a, 2a-b)$, thus $3\cdot 1 = 3a = 3b$. Since $\chr(\FF)\neq 3$, it follows that $a = b = 1$ and $\varphi^-(1) = 1$.

$\bullet$ Case 2: $\chr(\FF)= 3$. Due to Lemma \ref{lemma.Stab1PreserveCenter} we get that $\varphi$ restricts to the subpair $\mathcal{K}=(\cZ(\cA_\cR,\inv),\cZ(\cA_\cR,\inv))$. Denote $s=(\mathrm{I}_n,-\mathrm{I}_n)$, $e_1=(\mathrm{I}_n,0)$ and $e_2=(0,\mathrm{I}_n)$. Due to Corollary \ref{corollary.AutDim2}, we know that there are $\alpha\in \cR$ and $\beta\in \cR^{\times}$ such that $\varphi^+(1)=1$, $\varphi^+(s)=\alpha 1 + \beta s$, $\varphi^-(1)=1+\frac{\alpha}{\beta}s$, $\varphi^-(s)=\frac{1}{\beta}s$. Denote $r_1=-1-\frac{\alpha-1}{\beta}$, $r_2=-1+\frac{\alpha-1}{\beta}$, $t_1=-1-\frac{\alpha+1}{\beta}$ and $t_2=-1+\frac{\alpha+1}{\beta}$. Then we have $\varphi^-(e_1)=t_1e_1+t_2e_2$, $\varphi^-(e_2)=r_1e_1+r_2e_2$, $\varphi^+(e_1)=\beta t_1e_1-\beta t_2e_2$ and $\varphi^+(e_2)=-\beta r_1e_1+\beta r_2e_2$ and we have $r_1+r_2=1=t_1+t_2$.

Since $U_{\varphi^{\sigma}(e_1)}\varphi^{-\sigma}(\cA_\cR)=\varphi^{\sigma}(U_{e_1}(\cA_\cR)) = \varphi^{\sigma}(\cA_\cR e_1)$ for $\sigma=\pm$, using Lemma \ref{lemma.ImageOfUe1} we have:
\begin{align*}
& \cH(\cA_\cR,\inv)(t_1e_1+ t_2e_2)\subseteq U_{ t_1e_1+t_2e_2}(\cA_\cR) = U_{ \varphi^-(e_1)}\varphi^+(\cA_\cR) = \varphi^{-}(\cA_\cR e_1), \\
& \cS(\cA_\cR,\inv)(t_1e_1+ t_2e_2)\subseteq U_{\beta t_1e_1-\beta t_2e_2}(\cA_\cR) = U_{ \varphi^+(e_1)}\varphi^-(\cA_\cR) = \varphi^{+}(\cA_\cR e_1).
\end{align*}
In particular, since $(E_{11},E_{11})\in \cH(\cA_\cR,\inv)$ and $(E_{12},-E_{12})\in \cS(\cA_\cR,\inv)$ there exist $x,y\in \cA_\cR e_1$ such that $\varphi^{-}(x)=(t_1 E_{11},t_2 E_{11})$ and $\varphi^{+}(y)=(t_1 E_{12},-t_2 E_{12})$. 
Then:
\[0=\varphi^- \{x,y,x\}^{-}=(-t_1^2t_2E_{12},0).\]
This shows that $t_1^2 t_2=0$. Analogously we have $t_2^2t_1=r_1^2r_2=r_2^2r_1=0$. Therefore due to Lemma \ref{lemma.OrthogonalIdempotent} we have that $r_1^2=r_1$, $r_2^2=r_2$, $t_1^2=t_1$, $t_2^2=t_2$ and $r_1r_2=t_1t_2=0$.

Since $1=\varphi^+(1)=\varphi^+(e_1+e_2)=\beta(t_1-r_1)e_1+\beta(r_2-t_2)e_2$ we have:
\[\beta(t_1-r_1)=1_\cR=\beta(r_2-t_2).\]
Therefore:
\begin{align*}
& r_2=r_2\beta(t_1-r_1)=\beta r_2t_1=t_1\beta(r_2-t_2)=t_1, \\
& r_1=r_1\beta(r_2-t_2)=-\beta r_1t_2=t_2\beta(t_1-r_1)=t_2.
\end{align*}
Now:
\[\varphi^{-}(1)=\varphi^-(e_1+e_2)=r_1e_1+r_2e_2+t_1e_1+t_2e_2=r_1e_1+r_2e_2+r_2e_1+r_1e_2=e_1+e_2=1\]
\end{proof}

\begin{theorem} \label{automorphismschemes.unitar}
	Let $n \in \NN$, $\cA := \cM_n(\FF) \oplus \cM_n(\FF)^{op}$ and $\inv$ the exchange involution of $\cA$. Let $\cV$ be the Kantor pair associated to $(\cA, \inv)$. Then, there are isomorphisms of affine group schemes:
\begin{align}
	\bAut(\cV) &\simeq \big( \bG_m(\cA) \otimes_{\bG_m} \bIso(\cA,\inv) \big) \rtimes \bmu_2,
	\text{ in case } 1 < n \in \NN, \text{ or } n=1 \text{ and } \chr\FF\neq 3,\\
	\bAut(\cV)&\simeq \bGL_2(\FF) \text{ in case } n=1 \text{ and } \chr\FF= 3.
\end{align}
	
\end{theorem}
\begin{proof}
$\bullet$ Case 1: Let either $1 < n \in \NN$, or $n=1$ and $\chr\FF\neq 3$. Consider the morphism:
\begin{equation}\label{eq:thetaUnitary}\theta\colon (\bG_m(\cA)\times \bIso(\cA,\inv))\rtimes \bmu_2\longrightarrow \bAut(\cV)
\end{equation}
given by $\theta_{\cR}(a,b,r)= \widetilde{L}_a \widehat{R}_{\bar{b}}\circ \tau_r$ where \[\tau_r(X,Y)=(Y^t,X^t)\otimes\frac{1}{2}(1-r)+(X,Y)\otimes\frac{1}{2}(1+r).\]
	
	First, we have to prove that $\theta$ is an epimorphism. Take $\varphi\in \bAut(\cV)(\cR)$. Due to Proposition \ref{proposition.invertToInvert}, it follows that $a:=\varphi^+(1^+)$ is invertible. Consider the morphism $\psi=\widetilde{L}_{a^{-1}}\circ \varphi$. Clearly, $\psi\in \Stab_{\Aut_{\cR}(\cV_{\cR})}(1^+)=\Stab_{\Aut_{\cR}(\cV_{\cR})}(1^+,1^-)$ by Lemma~\ref{lemma.Stab1}. Hence by \cite[Proposition~3.1]{AC21} and the fact that either the algebra is generated by its hermitian elements (see Lemma \ref{lemma.HermitianSkewHermitianCubed}) or $\chr\FF\neq 3$, it follows that $\psi^+=\psi^-\in\bAut(\cA,\inv)$.

	Due to \eqref{eq:IsomUnitary} and \eqref{eq:unitaryIota} there is some automorphism $\widetilde{\varphi}$ of $\cM_{n}(\cR)$ and an element $r \in \bmu_2(\cR)$ such that  $\psi=\iota(\widetilde{\varphi})\circ \tau$ with $\tau = \tau_r$. Due to Remark \ref{rem.schemes} and the sheaf property of quotients \cite[\S~15.5]{W79} there is some faithfully flat extension $f\colon \cR \to \cS$ and some element $c\in \GL_n(\cS)$ such that $\bAut(\cM_n(\FF))(f)(\widetilde{\varphi})=\Ad_c$. Concretely, if $b=(c,c^{-1})\in \bIso(\cA,\ex)(\cS)$, it follows that $\iota( \bAut(\cM_n(\FF))(f)(\widetilde{\varphi}))=\Ad_b=L_bR_b^{-1}$ (see \eqref{eq:unitaryIota} for the definition de $\iota$). Thus, $\bAut(\cA,\inv)(f)(\psi^+)=L_bR_b^{-1}\circ \bmu_2(f)(\tau)$.  Denote by $\widetilde{\tau}=(\tau,\tau)$. Since $\psi^+=\psi^-$, it follows that $\bAut(\cV)(f)(\psi)=\widetilde{L}_b\widehat{R}_{b^{-1}}\circ \widetilde{\tau}$. Since $\bAut(\cV)(f)(\widetilde{L}_{a^{-1}})=\widetilde{L}_{\widetilde{a}^{-1}}$ where $\widetilde{a}=\bG_m(\cA(f)(a))$, it follows that $\bAut(\cV)(f)(\varphi)=\bAut(\cV)(f)(\widetilde{L}_{a})\circ \widetilde{L}_b\widehat{R}_{b^{-1}}\circ \widehat{\tau}= \widetilde{L}_{\widetilde{a}b}\widehat{R}_{b^{-1}}\circ \widehat{\tau} =\theta_{\cS}(\widetilde{ab},\overline{b}^1,s)$ for some $s$ such that $\widehat{\tau}=\tau_s$. Therefore, due to the sheaf property of quotient maps \cite[\S~15.5]{W79}, $\theta$ is an epimorphism.
	
	If $(a,b,r)\in \ker(\theta_{\cR})$, then, it is straightforward to show that $r=1$. And in a similar way as in Proposition  \ref{automorphismschemes}, if we write $a=(a_1,a_2)$ and $b=(b_1,b_2)$, from the condition that $\widetilde{L}_a\widetilde{R}_{\overline{b}}=\id$, we get that there are $\lambda,\beta\in \cR^{\times}$ such that $a_1=b_2^{-1}=\lambda$ and $a_2=b_1^{-1}=\beta$ from the fact that $b\in\bIso(\cA,\inv)(\cR)$, it follows that $\lambda=\beta^{-1}$. Therefore, $\mathbf{ker}(\theta)(\cR)=\{\big((r,r^{-1}),(r,r^{-1}),1\big)\mid r\in \cR^{\times}\} \simeq \bG_m(\cR)$.
	
$\bullet$ Case 2: Let $n=1$ and $\chr\FF=3$. By Lemma~\ref{lemma.AutKantor}, the automorphisms of $\cV_{\cR}$ are those of the form $(\varphi^+,\varphi^-)$ where $\varphi^+\in\GL_2(\cR)$ and $\varphi^-$ is the dual inverse of $\varphi^+$ with respect to $n$.
\end{proof}

\begin{corollary}
Let $\chr \FF\neq 3$, $\cA:=\FF\oplus \FF$ and $\inv$ the exchange involution. Let $\cV$ be the Kantor pair associated to $(\cA,\inv)$. Then 
\begin{equation}
\bAut(\cV)\cong \bG_m(\cA)\rtimes \bmu_2.
\end{equation}
\end{corollary}
\begin{proof}
Consider the morphism
\[\gamma\colon (\bG_m(\cA)\times \bIso(\cA,\inv))\rtimes \bmu_2\longrightarrow \bG_m(\cA)\rtimes \bmu_2\]
given by $\gamma_{\cR}(a,b,r)=(ab^{-1},r)$ for every associative commutative and unital algebra $\cR$. It is clearly an epimorphism since it is surjective for every $\cR$.
		
Moreover, $\mathbf{ker}(\gamma)(\cR)=\{\big((r,r^{-1}),(r,r^{-1}),1\big)\mid r\in \cR^{\times}\} \simeq \bG_m(\cR)$, which is the same kernel of the morphism $\theta$ defined in \eqref{eq:thetaUnitary}. Thus, we have the isomorphism.
\end{proof}

\begin{corollary}
Let $ n \in \NN$, $\cA := \cM_n(\FF) \oplus \cM_n(\FF)^{op}$ and $\inv$ the exchange involution of $\cA$. Let $\cT$   be the Kantor triple system associated to $(\cA, \inv)$. Then, there are isomorphisms of affine group schemes
\begin{align}
		\bAut(\cT) &\simeq\big( \bIso(\cA,\inv) \otimes_{\bG_m} \bIso(\cA,\inv) \big) \rtimes \bmu_2
		\text{ in case either } 1 < n \in \NN, \text{ or } n=1 \text{ and } \chr\FF\neq 3,\\
		\bAut(\cT)&\simeq \bOrt_2 \text{ in case } n=1 \text{ and } \chr\FF= 3.
	\end{align}
\end{corollary}
\begin{proof}
Let $\cV$  be the Kantor triple system associated to $(\cA, \inv)$.

$\bullet$ Case 1: Let $ 1 < n \in \NN$  or  $n=1$  and  $\chr\FF\neq 3$. Consider the morphism:
\[\gamma\colon (\bIso(\cA,\inv)\times \bIso(\cA,\inv))\rtimes \bmu_2\longrightarrow \bAut(\cT)\]
given by $\gamma_{\cR}(a,b,r)= L_a R_{\bar{b}}\circ \tau_r$ where
\[\tau_r(X,Y)=(Y^t,X^t)\otimes\frac{1}{2}(1-r)+(X,Y)\otimes\frac{1}{2}(1+r).\]
	
	In order to prove that it is an epimorphism, we notice that given $\varphi\in\bAut(\cT)(\cR)$, then $(\varphi,\varphi)\in \bAut(\cV)(\cR)$. Therefore, due to the fact that $\theta$ as defined in \eqref{eq:thetaUnitary} is an epimorphism and the sheaf property of quotient maps, there is a faithfully flat extension $\cR\to \cS$ such that if we embed $(\varphi,\varphi)$ into $\bAut(\cT)(\cS)$, we get that  $(\varphi,\varphi)=\widetilde{L}_a\widetilde{R}_{\overline{b}}\circ \tau_r$ for some $a\in\bG_m(\cA)(\cS)$ and $b\in \bIso(\cA,\inv)(\cS)$ and $r\in\bmu_2(\cS)$. Since the positive and the negative part are equal, we get that $L_aR_{\overline{b}}\circ \tau_r=L_{\overline{a}}^{-1}R_{\overline{b}}\circ \tau_r$. Therefore, we have that $L_aR_{\overline{b}} =L_{\overline{a}}^{-1}R_{\overline{b}}$ and if we apply this to $1$, we get $a\overline{b}=\overline{a}^{-1}\overline{b}$. Thus, we get $\overline{a}a=b\overline{b}=1$. Thus $\varphi=L_aR_{\overline{b}}\circ\tau_r$. The sheaf property of quotient maps shows that this is an epimorphism.
	
	The kernel is $\bG_m$ in the same way as in Theorem~\ref{automorphismschemes.unitar}.

$\bullet$ Case 2: Let $n=1$  and  $\chr\FF= 3$. Let $n$ be the polar form of the norm of the Hurwitz algebra $\cA$.
Note that for each $\cR$ and $\varphi \in \GL_2(\cR)$, we have $\varphi\in\bAut(\cT)(\cR)$ if and only if $(\varphi,\varphi)\in\bAut(\cV)(\cR)$. Then, from Lemma~\ref{lemma.AutKantor} it follows that $\varphi\in\bAut(\cT)(\cR)$ if and only if $\varphi$ is its own dual inverse relative to the bilinear form $n$.
\end{proof}


\begin{proposition} \label{orbits} \;
\begin{itemize}
\item[$1)$] Let $n \in \NN$, $\cA := \cM_n(\FF)$ and $\inv$ the transposition or the standard symplectic involution of $\cA$. Let  $\cV$ be the Kantor pair associated to $(\cA, \inv)$. Then, the orbits on $\cV^+$ are exactly the sets
$ \cO_{i} := \{ x \in \cA \med  \rank(x)= i \},$
with $0 \leq i \leq  n$.
\item[$2)$] Let $ n \in \NN$, $\cA := \cM_n(\FF) \oplus \cM_n(\FF)^{op}$ and $\inv$ the exchange involution of $\cA$.  Let $\cV$ be the Kantor pair associated to $(\cA, \inv)$. \\
$\bullet$ If $n > 1$ or $\chr\FF\neq3$, then the orbits on $\cV^+$ are exactly the sets
$$ \cO_{i,j} := \{ (x,y) \in \cA \med \{ \rank(x), \rank(y) \} = \{ i, j \} \},$$
with $0 \leq i \leq j \leq n$. \\
$\bullet$ If $n = 1$ and $\chr\FF=3$, then the orbits on $\cV^+$ are exactly the sets
$$ \cO_0 := \{0\}, \quad \cO_1 := \cA \setminus \{0\}.$$
\end{itemize}
\end{proposition}
\begin{proof}
1) Since the morphism $\Phi$ from \eqref{eq:IsoAutV} is an epimorphism, it follows that $\Phi_{\FF}$ is surjective \cite[Theorem A.48]{EKmon}, thus, every automorphism of $\cV$ is of the form $\widetilde{L}_a\widehat{R}_b$ with $a\in \cA^{\times}$ and $b\in \Sim(\cA,\inv)$. From here the result follows trivially.

2) If $n=1$ and $\chr\FF=3$ it is clear since every endomorphism of $\FF\oplus\FF$ is the positive part of an automorphism of $\cV$. Otherwise, as in 1), since the morphism $\theta$ from \eqref{eq:thetaUnitary} is an epimorphism, the automorphisms of $\cV$ are of the form $\widetilde{L}_a\widehat{R}_b\circ \tau$ or $\widetilde{L}_a\widehat{R}_b$,
with $a\in \cA^{\times}$, $b\in \Iso(\cA,\inv)$ and $\tau(X,Y)=(Y^t,X^t)$. Thus, the result follows trivially.
\end{proof}


\begin{proposition}
Let $(\cA, \inv)$ be a finite-dimensional associative structurable algebra, and $\frg = \kan(\cA, \inv) = \bigoplus_{i=-2}^2 \frg_i$ the $\ZZ$-graded ($5$-graded) Lie algebra given by the Kantor construction. Let $\cV$ be the Jordan pair associated to the $\ZZ$-graded Lie algebra $\cL = \bigoplus_{i=-1}^1 \cL_i \leq \frg$, where $\cL_{1} := \frg_2$, $\cL_{-1} := \frg_{-2}$, $\cL_0 := \lspan\{ [\frg_{-2}, \frg_2] \} \leq \frg_0$. For the below cases 1) and 3), we also assume that there exists a square root $\mathbf{i} \in \FF$ of $-1$. Then:
\begin{itemize}
\item[$1)$]
If $\cA = \cM_n(\FF)$ and the involution is the transposition $\bar x := x^\Tr$, then we have
$\cV \cong \VII_n$, where $\VII_n$ denotes a simple Jordan pair of type $\textnormal{II}_n$.
\item[$2)$] 
If $\cA = \cM_n(\FF)$ with $n = 2m$ and the involution is the standard symplectic involution $\bar x := x^\s$, then we have $\cV \cong \VIII_n$, where $\VIII_n$ denotes a simple Jordan pair of type $\textnormal{III}_n$.
\item[$3)$]
If $\cA = \cM_n(\FF)\times \cM_n(\FF)^{op}$ and the involution is given by  $\overline{(x,y)}= \ex(x,y) := (y,x)$, then we have $\cV \cong \VI_n$, where $\VI_n$ denotes a simple Jordan pair of type $\textnormal{I}_n$.
\end{itemize}
\end{proposition}
\begin{proof}
			Let $\cS = \cS(\cM_n(\FF), \inv)$. From the Kantor construction we know that
			$\cV^\sigma = \{\iota^\sigma(L_x) \med x\in \cS\}$ for $\sigma = \pm$, where
			$\iota^\sigma \colon \lspan\{ L_x \med x\in \cS \}
			\subseteq \End(\cA) \longrightarrow \frg_{\sigma 2} \subseteq \kan(\cA, \inv)$
			are the inclusions given by
			\begin{equation*}
				\iota^+(f) := \begin{pmatrix} 0 & 0 \\ f & 0 \end{pmatrix}, \quad
				\iota^-(f) := \begin{pmatrix} 0 & f \\ 0 & 0 \end{pmatrix}.
			\end{equation*}
			The triple products of $\cV$ are given by
			\begin{equation*}
				\{ \iota^\sigma(L_x), \iota^{-\sigma}(L_y), \iota^\sigma(L_z) \}
				= \big[ [ \iota^\sigma(L_x), \iota^{-\sigma}(L_y) ], \iota^\sigma(L_z) \big]
				= \iota^\sigma(L_xL_yL_z + L_zL_yL_x) = \iota^\sigma(L_{xyz + zyx}),
			\end{equation*}
			which are also determined by the quadratic products
			$Q_{\iota^\sigma(L_x)}\big( \iota^{-\sigma}(L_y) \big) = \iota^\sigma(L_{xyx})$, where $Q_x(y) := \frac{1}{2} \lbrace x,y,x \rbrace$.
			If we identify $\iota^\sigma(L_x) \equiv x$ and $\cV^\sigma \equiv \cS$, then we can write $Q_x(y) = xyx$. We can now particularize on each case.
			
			$1)$ Note that $\cV^\sigma = \cS = A_n(\FF) := \cS(\cM_n(\FF), \Tr)$, and recall from \cite[Chapter~4]{L75} that $\VII_n := (A_n(\FF), A_n(\FF))$ has triple products $\{x,y,z\} := x y^\Tr z + z y^\Tr x = -(xyz+zyx)$. Then it is clear that we have an isomorphism $\VII_n \longrightarrow \cV$ given by
			\begin{equation}
				(\VII_n)^\sigma \longrightarrow \cV^\sigma, \qquad x \longmapsto \mathbf{i} x.
			\end{equation}
			
			$2)$ Recall that $\VIII_n := (H_n(\FF), H_n(\FF))$, with $H_n(\FF) := \cH(\cM_n(\FF), \Tr)$, has triple products $\{x,y,z\} := x y^\Tr z + z y^\Tr x = xyz+zyx$, and the associated quadratic products are $Q_x(y) = xyx$. In this case,
			$$\cV^\sigma = \cS = \left\{ \begin{pmatrix} A & B \\ C & -A^\Tr \end{pmatrix}
			\med A \in \cM_m(\FF), \; B,C \in H_m(\FF) \right\}.$$
			For $\sigma = \pm$, consider the linear isomorphisms
			\begin{align}
				\varphi^\sigma \colon \cV^\sigma = \cS \longrightarrow (\VIII_n)^\sigma = H_n(\FF), \quad
				\varphi^+ \begin{pmatrix} A & B \\ C & -A^\Tr \end{pmatrix}
				:= \begin{pmatrix} B & A \\ A^\Tr & -C \end{pmatrix}, \quad
				\varphi^- \begin{pmatrix} A & B \\ C & -A^\Tr \end{pmatrix}
				:= \begin{pmatrix} C & A^\Tr \\ A & -B \end{pmatrix}.
			\end{align}
			Fix $x = \begin{pmatrix} A & B \\ C & -A^\Tr \end{pmatrix}$,
			$y = \begin{pmatrix} \alpha & \beta \\ \gamma & -\alpha^\Tr \end{pmatrix}$ in $\cS$. Then
			$\varphi^+(x) = \begin{pmatrix} B & A \\ A^\Tr & -C \end{pmatrix}$,
			$\varphi^-(y) = \begin{pmatrix} \gamma & \alpha^\Tr \\ \alpha & -\beta \end{pmatrix}$,
			and
			\begin{align*}
				\varphi^+ \big( Q_x(y) \big) &= \varphi^+ \begin{pmatrix} 
					A \alpha A + B \gamma A + A \beta C - B \alpha^\Tr C &
					A \alpha B + B \gamma B - A \beta A^\Tr + B \alpha^\Tr A^\Tr \\
					C \alpha A - A^\Tr \gamma A + C \beta C + A^\Tr \alpha^\Tr C &
					C \alpha B - A^\Tr \gamma B - C \beta A^\Tr - A^\Tr \alpha^\Tr A^\Tr \end{pmatrix} \\
				&= \begin{pmatrix}
					A \alpha B + B \gamma B - A \beta A^\Tr + B \alpha^\Tr A^\Tr &
					A \alpha A + B \gamma A + A \beta C - B \alpha^\Tr C \\
					A^\Tr \alpha^\Tr A^\Tr + A^\Tr \gamma B + C \beta A^\Tr - C \alpha B &
					- C \alpha A + A^\Tr \gamma A - C \beta C - A^\Tr \alpha^\Tr C \end{pmatrix} \\
				&= \varphi^+(x) \varphi^-(y) \varphi^+(x) = Q_{\varphi^+(x)} \big( \varphi^-(y) \big).
			\end{align*}
			Since $\varphi^- = \varphi^+ \circ \Tr$, we also get
			\begin{align*}
				\varphi^-\big( Q_y(x) \big) = \varphi^+ \big( (yxy)^\Tr \big)
				= \varphi^+( y^\Tr x^\Tr y^\Tr ) = \varphi^+ \big( Q_{y^\Tr}(x^\Tr) \big)
				= Q_{\varphi^+(y^\Tr)} \big( \varphi^-(x^\Tr) \big)
				= Q_{\varphi^-(y)} \big( \varphi^+(x) \big),
			\end{align*}
			and therefore $\varphi$ is an isomorphism of Jordan pairs.
	
            $3)$ Note that $\cV^\sigma = \cS := \cS(\cM_n(\FF)\times \cM_n(\FF)^{op}, \ex)= \lbrace (x,-x) \mid x\in \cM_n(\FF) \rbrace \cong \cM_n(\FF)$.

             and recall  that $\VI_n := (\cM_n(\FF), \cM_n(\FF))$ has triple products $\{x,y,z\}^{\sigma} := x y z + z y x$. Then it is clear that we have an isomorphism $\VI_n \longrightarrow \cV$ given by
			\begin{equation}
				(\VI_n)^\sigma \longrightarrow \cV^\sigma, \qquad x \longmapsto (\sigma\mathbf{i}x, -\sigma\mathbf{i} x).
			\end{equation}	
			
			
\end{proof}

\section{Some tensor product decompositions}\label{s:section2}

In this section we will describe some tensor product decompositions for some of the Kantor pairs studied in this work, regarded as metric generalized Jordan (super)pairs, and we also give some examples for other metric generalized Jordan (super)pairs. These decompositions are related to the corresponding decompositions of the automorphism group schemes as central products of smaller group schemes for the Kantor pairs of the orthogonal and symplectic cases. Due to the lack of enough theory (for tensor products of metric generalized Jordan superpairs), we still cannot use these tensor product decompositions to deduce the automorphism group schemes, which would simplify the proofs for the orthogonal and symplectic cases.

For the classification and definitions of the simple Jordan pairs (and some of their invariants, like their generic trace), the reader may consult \cite[Chapter~IV]{L75}. For the construction and automorphisms of the simple Jordan antipairs, some important references are \cite{FF80} and \cite{B08}.

\begin{df}
For homogeneous elements $x,y,z$ in a $\ZZ_2$-graded vector space, we will denote
\begin{align}
\eta_{x,y} &:= (-1)^{\varepsilon(x)\varepsilon(y)}, \\
\eta_{x,y,z} &:= (-1)^{\varepsilon(x)\varepsilon(y)
+ \varepsilon(y)\varepsilon(z) + \varepsilon(z)\varepsilon(x)},
\end{align}
where $\varepsilon$ is the degree map (called \emph{parity}).
A \emph{(trilinear) pair} is a pair of vector spaces $\cV = (\cV^+, \cV^-)$ with a pair of trilinear maps $\{\cdot,\cdot,\cdot\}^\sigma \colon \cV^\sigma \times \cV^{-\sigma} \times \cV^\sigma \to \cV^\sigma$, $\sigma \in \{+, -\}$. The $D$-operators are defined by $D^\sigma_{x,y}(z) := \{x,y,z\}^\sigma$. A \emph{(trilinear) superpair} is a trilinear pair where $\cV^-$ and $\cV^+$ are $\ZZ_2$-graded. For homogeneous $x,y$, the maps $D^\sigma_{x,y}$ are homogeneous of parity $\varepsilon(D^\sigma_{x,y}) := \varepsilon(x) + \varepsilon(y)$. A \emph{generalized Jordan superpair} is a trilinear superpair $\cV$ where we have that
\begin{equation} \label{defGJSP}
[D^\sigma_{x,y}, D^\sigma_{z,w}] = 
D^\sigma_{D^\sigma_{x,y}z, w} - \eta_{x,y,z} D^\sigma_{z, D^{-\sigma}_{y,x}w},
\end{equation}
for any homogeneous elements $x,z\in \cV^\sigma$, $y,w\in \cV^{-\sigma}$, $\sigma = \pm$, where the left side of \eqref{defGJSP} denotes the Lie superbracket. Even superpairs ($\cV^\sigma_{\bar1} = 0$ for $\sigma = \pm$) are called \emph{pairs} and odd superpairs ($\cV^\sigma_{\bar0} = 0$ for $\sigma = \pm$) are called \emph{antipairs}.

Let $\cV$ be a generalized Jordan superpair with a bilinear form $\langle\cdot,\cdot\rangle \colon \cV^- \times \cV^+ \to \FF$. We say that $\langle\cdot,\cdot\rangle$ is \emph{superinvariant} if for any homogeneous $x,z\in\cV^-$, $y,w\in\cV^+$ we have
\begin{equation} \label{superinvariant}
\langle D_{x,y}z, w \rangle = \eta_{x,y,z} \langle z, D_{y,x}w \rangle.
\end{equation}
We say that $\langle\cdot,\cdot\rangle$ is \emph{even} if $\langle x, y \rangle = 0$ for any homogeneous $x\in\cV^-$, $y\in\cV^+$ such that $\varepsilon(x) \neq \varepsilon(y)$. 
We say that $\langle\cdot,\cdot\rangle$ is \emph{supersymmetric} if for any homogeneous $x,z\in\cV^-$, $y,w\in\cV^+$ we have
\begin{equation}\label{symmetricForm}
\begin{split}
\langle D_{x,y}z, w \rangle &= \eta_{D_{x,y}, D_{z,w}} \langle D_{z,w}x, y \rangle, \\
\langle x, D_{y,z}w \rangle &= \eta_{D_{x,y}, D_{z,w}} \langle z, D_{w,x}y \rangle.
\end{split}
\end{equation}
We say that $(\cV, \langle\cdot,\cdot\rangle)$ is a \emph{metric generalized Jordan superpair} if $\cV$ is a generalized Jordan superpair and $\langle\cdot,\cdot\rangle \colon \cV^- \times \cV^+ \to \FF$ is a nondegenerate even superinvariant supersymmetric bilinear form; in that case, $\langle\cdot,\cdot\rangle$ is called the \emph{metric} of $\cV$. We may denote $\langle\cdot,\cdot\rangle^- := \langle\cdot,\cdot\rangle$ and, by the rule of signs, we define the map $\langle\cdot,\cdot\rangle^+ \colon \cV^+ \times \cV^- \to \FF$, $\langle x,y \rangle^+ := \eta_{x,y} \langle y,x \rangle^-$.
\end{df}

\begin{notation}
Let $\cV$ be a generalized Jordan superpair with a metric $\langle\cdot,\cdot\rangle^\sigma \colon \cV^\sigma \times \cV^{-\sigma} \longrightarrow \FF$. For $a \in \ZZ_2$, denote $\eta_a := (-1)^a$. Given $\alpha = (\lambda,a) \in \FF \times \ZZ_2$, the \emph{tensor-shift} by $\alpha$ of $(\cV, \langle\cdot,\cdot\rangle)$ (see \cite[Notation~4.6]{A22}), is the metric generalized Jordan superpair $\cV^{[\alpha]} := \cV$ with shifted degree $\varepsilon_{[\alpha]}(x) := \varepsilon(x)+a$, metric
\begin{equation}
\begin{split}
& \langle x^+,y^- \rangle_{[\alpha]}^+ := \eta_{a,x} \langle x^+,y^- \rangle^+ = \eta_{a,y} \langle x^+,y^- \rangle^+, \\
& \langle x^-,y^+ \rangle_{[\alpha]}^- := \eta_a\eta_{a,x} \langle x^-,y^+ \rangle^- = \eta_a\eta_{a,y} \langle x^-,y^+ \rangle^-.
\end{split}
\end{equation}
and triple products
\begin{equation}
\begin{split}
\{x,y,z\}_{[\alpha]}^+ &:= \eta_{a,y}(\{x,y,z\}^+ + \lambda \langle x,y \rangle^+ z), \\
\{x,y,z\}_{[\alpha]}^- &:= \eta_a\eta_{a,y}(\{x,y,z\}^- + \lambda \langle x,y \rangle^- z). \\
\end{split}
\end{equation}
(Notice that the parities of the elements in the right side of the equalities above correspond to the ones in the original pair before doing the tensor-shift.)
\end{notation}

\begin{notation}
Recall from \cite[Proposition~4.3]{AC26} (or \cite[Proposition~4.3]{A22}) that the \emph{tensor (super)product} of metric generalized Jordan (super)pairs $\cV_i$ for $i=1,...,n$ is given by $\cW := \bigotimes_{i=1}^n \cV_i$ with bilinear form $\langle \cdot, \cdot \rangle$  given by
\begin{align}
\langle \otimes_i f_i, \otimes_i v_i \rangle
&= \prod_{i=1}^n \Bigl( \prod_{k < i} \eta_{f_i, v_k} \Bigr)
\langle f_i, v_i \rangle
= \Bigl( \prod_{1 \leq j < i \leq n} \eta_{f_i, v_j} \Bigr)
\Bigl( \prod_{i=1}^n \langle f_i, v_i \rangle \Bigr),
\end{align}
and triple products $\cW^\sigma \times \cW^{-\sigma} \times \cW^\sigma \to \cW^\sigma$,
\begin{equation}\begin{aligned}
\{\otimes_i x_i, \otimes_i y_i, \otimes_i z_i \}
&= \Bigl( \prod_{1 \leq j < i \leq n} \eta_{x_i, y_j} \Bigr)
\sum_{i=1}^n \Bigl( \prod_{t < i} \eta_{x_i, z_t} \eta_{y_i, z_t} \Bigr) \cdot \\
& \qquad\cdot \langle x_1, y_1 \rangle z_1 \otimes\dots\otimes \{x_i, y_i, z_i\}
\otimes\dots\otimes \langle x_n, y_n \rangle z_n.
\end{aligned}\end{equation}
Also, recall that tensor-shifts are just the tensor products by $1$-dimensional (super)pairs.
\end{notation}

\begin{examples} \label{Examples1} \;
\begin{itemize}
\item[1)]
Let $\cV = (\cV^+, \cV^-) = \cV^{\text{(I)}}_{m,n} := (\cM_{m,n}(\FF), \cM_{n,m}(\FF))$ denote the simple Jordan pair of type I with parameters $m,n$. Its triple products (using the construction from \cite{S84}) are
\begin{equation}
\{ x,y,z \}^\sigma = xyz + zyx.
\end{equation}
Recall that the generic trace of $\cV$ (which is a metric by \cite[Example~4.7]{A22}) is
\begin{equation}
t^\sigma \colon \cV^\sigma \times \cV^{-\sigma} \longrightarrow \FF,
\qquad  t^\sigma(x,y) := \tr(xy) = \tr(yx).
\end{equation}
In particular,
\begin{equation} \label{ctes.base.Eij.par.I}
\{ E_{ri}, E_{sj}, E_{tk} \}^\sigma = \delta_{is} \delta_{jt} E_{rk} + \delta_{ks} \delta_{jr} E_{ti},
\qquad  t^\sigma(E_{ri},E_{sj}) = \delta_{is}\delta_{rj}.
\end{equation}
\\ $\bullet$ More in particular, for the pair $\cV = \cV^{\text{(I)}}_{1,n} \equiv (\FF^n, \FF^n)$ we can identify both bases $\{ E_{1i} \}$ of $\cV^+$ and $\{ E_{i1} \}$ of $\cV^-$ with the canonical basis $\{ e_i \}$ of $\FF^n$, and then the triple products and generic trace can be written as
\begin{equation} \label{ctes.base.ei.par.I}
\{ e_i, e_j, e_k \}^\sigma = \delta_{ij} e_k + \delta_{kj} e_i,
\qquad  t^\sigma(e_i^\sigma, e_j^{-\sigma}) = \delta_{ij}.
\end{equation}

\item[2)]
Let $\cA = (\cA^+, \cA^-) = \cA^{\text{(I)}}_{m,n} := (\cM_{m,n}(\FF), \cM_{n,m}(\FF))$ be the simple Jordan antipair of type I and parameters $m,n$ (e.g., see \cite[Example~1.2.3]{B08}). Its triple products are
\begin{equation}
\{ x,y,z \}^\sigma = xyz - zyx.
\end{equation}
Note that a metric $\widetilde{t}$ of $\cA$ is given by 
\begin{align}
\widetilde{t}^\sigma \colon \cA^\sigma \times \cA^{-\sigma} \longrightarrow \FF,
\qquad \widetilde{t}^\sigma(x,y) := -\sigma t^\sigma(x,y) = -\sigma \tr(xy) = -\sigma \tr(yx).
\end{align}
In particular,
\begin{equation} \label{ctes.base.Eij.antipar.I}
\{ E_{ri}, E_{sj}, E_{tk} \}^\sigma = \delta_{is} \delta_{jt} E_{rk} - \delta_{ks} \delta_{jr} E_{ti},
\qquad  \widetilde{t}^\sigma(E_{ri},E_{sj}) = -\sigma\delta_{is}\delta_{rj}.
\end{equation}
(Although it is easy to show that $\widetilde{t}$ is a metric, that also follows from isomorphism \eqref{decomp.antipair.Imn} below.)
\\ $\bullet$ More in particular, for the antipair $\cA = \cA^{\text{(I)}}_{1,n} \equiv (\FF^n, \FF^n)$ we can identify both bases $\{ E_{1i} \}$ of $\cA^+$ and $\{ E_{i1} \}$ of $\cA^-$ with the canonical basis $\{ e_i \}$ of $\FF^n$, and then the triple products and metric can be written as
\begin{equation} \label{ctes.base.ei.antipar.I}
\{ e_i, e_j, e_k \}^\sigma = \sigma (\delta_{ij} e_k - \delta_{kj} e_i),
\qquad  \widetilde{t}^\sigma(e_i, e_j) = -\sigma \delta_{ij}.
\end{equation}

\item[3)]
Recall that the metrics are always required to satisfy the rule of signs (which is used in the Faulkner construction, and therefore to define the triple products of tensor products of superpairs); in particular, we have $\langle y,x \rangle^- = \langle x,y \rangle^+$ for pairs (even superpairs), and $\langle y,x \rangle^- = -\langle x,y \rangle^+$ for antipairs (odd superpairs).

By \eqref{ctes.base.ei.par.I}, it is straightforward to see that for $\alpha = (-2, \bar 1)$, the antipair $(\cV^{\text{(I)}}_{1,n}, t)^{[\alpha]}$ has operations
\begin{align*}
\{e_i,e_j,e_k\}^\sigma_{[\alpha]} = -\sigma(\delta_{ij}e_k - \delta_{kj}e_i), \qquad
t^\sigma_{[\alpha]}(e_i,e_j) = \sigma\delta_{ij}.
\end{align*}
Now assume that there exists a square root $\bi$ of $-1$ in $\FF$. Then, by \eqref{ctes.base.ei.antipar.I}, it is easy to see that $\bi \id$ defines an isomorphism of metric antipairs
\begin{equation} \label{isom.superpairs.Imn}
(\cA^{\text{(I)}}_{1,n}, \widetilde{t}) \cong (\cV^{\text{(I)}}_{1,n}, t)^{[(-2, \bar 1)]},
\end{equation}
or equivalently, $(\cV^{\text{(I)}}_{1,n}, t) \cong (\cA^{\text{(I)}}_{1,n}, \widetilde{t})^{[(2, \bar 1)]}$. Therefore, $(\cV^{\text{(I)}}_{1,n}, t)$ and $(\cA^{\text{(I)}}_{1,n}, \widetilde{t})$ are tensor-shift related. Unfortunately, $(\cV^{\text{(I)}}_{m,n}, t)$ and $(\cA^{\text{(I)}}_{m,n}, \widetilde{t})$ do not seem to be tensor-shift related in general. However, we claim that we can decompose
\begin{equation} \label{decomp.antipair.Imn}
(\cA^{\text{(I)}}_{m,n}, \widetilde{t}) \cong
\big( (\cV^{\text{(I)}}_{1,m}, -t) \otimes (\cV^{\text{(I)}}_{1,n}, t) \big)^{[(0,\bar1)]}.
\end{equation}
To prove the isomorphism, first note that the triple products of
$\cV = (\cV^{\text{(I)}}_{1,m}, -t) \otimes (\cV^{\text{(I)}}_{1,n}, t)$
are given by
\begin{align*}
& \{ e_r \otimes e_i, e_s \otimes e_j, e_t \otimes e_k \}^\sigma
= \{ e_r,e_s,e_t \} \otimes \langle e_i,e_j \rangle e_k
  + \langle e_r,e_s \rangle e_t \otimes \{ e_i,e_j,e_k \} \\
& \qquad = (\delta_{rs}e_t + \delta_{ts}e_r) \otimes \delta_{ij}e_k
  - \delta_{rs}e_t \otimes (\delta_{ij}e_k + \delta_{kj}e_i) \\
& \qquad = \delta_{ts}\delta_{ij} e_r \otimes e_k - \delta_{rs}\delta_{kj} e_t \otimes e_i.
\end{align*}
Denote $\varepsilon_{ij}^+ := e_i \otimes e_j \in \cV^+$, $\varepsilon_{ij}^- := e_j \otimes e_i \in \cV^-$.
Then the calculations above show that
$$ \{ \varepsilon_{ri}^\sigma, \varepsilon_{sj}^{-\sigma}, \varepsilon_{tk}^\sigma \}^\sigma
= \sigma (\delta_{is}\delta_{jt}\varepsilon_{rk}^\sigma
  - \delta_{ks}\delta_{jr}\varepsilon_{ti}^\sigma). $$
The metric of $\cV$ is the tensor product $(-t) \otimes t$, which is given by
$$ \langle \varepsilon_{ri}^\sigma, \varepsilon_{sj}^{-\sigma} \rangle^\sigma
= -\delta_{is}\delta_{rj}. $$
It follows that the tensor-shift $\cV^{[(0,\bar1)]}$ is determined by
\begin{align*}
\{ \varepsilon_{ri}^\sigma, \varepsilon_{sj}^{-\sigma}, \varepsilon_{tk}^\sigma \}^\sigma
= \delta_{is}\delta_{jt}\varepsilon_{rk}^\sigma
  - \delta_{ks}\delta_{jr}\varepsilon_{ti}^\sigma, \qquad
\langle \varepsilon_{ri}^\sigma, \varepsilon_{sj}^{-\sigma} \rangle^\sigma
= -\sigma\delta_{is}\delta_{rj},
\end{align*}
and then it is clear that the maps $\varepsilon_{ij}^\sigma \mapsto E_{ij}^\sigma$ define an isomorphism as in \eqref{decomp.antipair.Imn}.
\end{itemize}
\end{examples}

\begin{example} \label{Examples2} \; \\
Let $b \colon \FF^n \times \FF^n \to \FF$ be the standard scalar product and $\cV_n^\text{(IV)} := (\FF^n, \FF^n)$ the simple Jordan pair of type IV associated to $b$. Also consider the canonical basis $\{e_i\}$ of $\FF^n$.

On the other hand, consider the structurable algebra $\cA := \cM_n(\FF)$ with the transposition involution, and let $\cV_\cA$ be the associated Kantor pair. Consider the bilinear form $t \colon \cA \times \cA \to \FF$, $t(x,y) := t(xy^{\tr})$.

As in Example~\ref{Examples1}-1), let $(\cV_{1,n}^\text{(I)}, t)$ denote the Jordan pair of type $\text{I}_{1,n}$, with its generic trace $t$ (which is a metric), and its canonical basis $\{e_i\}$.

Finally, recall from \cite[Example~4.8]{A22} that both $(\cV_n^\text{(IV)}, b)$ and $(\cV_\cA, t)$ are metric generalized Jordan pairs too, and we have an isomorphism of metric generalized Jordan pairs
\begin{align}
\Big( (\cV_{1,n}^\text{(I)}, t)
\otimes (\cV_n^\text{(IV)}, b) \Big)^{[-2]}
\longrightarrow (\cV_\cA, t),
\qquad e_i \otimes e_j \longmapsto E_{ij},
\end{align}
where the first tensor product is tensor-shifted by $\alpha = (-2, \bar0) \equiv -2$. Recall that tensor-shifts preserve the automorphism group schemes with metric, that is, $\bAut\Big((\cV, t)^{[\alpha]}\Big) = \bAut(\cV, t)$.

Of course, this tensor product decomposition is related to the decomposition of $\bAut(\cV_\cA, t) = \bAut(\cV_\cA) \simeq \bG_m(\cA) \otimes_{\bG_m} \bSim(\cA, \tr) = \bGL_n \otimes_{\bG_m} \bGO_n$ (see Theorem~\ref{automorphismschemes}), which is a central product of the group schemes 
$\bAut(\cV_{1,n}^\text{(I)}, t) = \bAut(\cV_{1,n}^\text{(I)}) \simeq \bGL_n$ (see \cite[Proposition~4.7]{AD25})
and
$\bAut(\cV_n^\text{(IV)}, b) = \bAut(\cV_n^\text{(IV)}) \simeq \bGO_n$ (see \cite[Proposition~3.3]{AD25}).
\end{example}

\begin{examples} \label{Examples3} \;
\begin{itemize}
\item[1)]
Recall from \cite[Example~1.2.11]{B08} that the symplectic anti-Jordan pair of dimension $2n$ is given by $\cA^{\text{(III)}}_{2n} := (T, T)$ where $T$ is a vector space of dimension $2n$, and the triple products are
\begin{align}
\{x,y,z\}^\sigma := b(x,y)z + b(y,z)x + b(x,z)y,
\end{align}
where $b \colon T \times T \to \FF$ is a non-degenerate alternating bilinear form. We will regard $T$ as an odd vector superspace, which means that $b$ is supersymmetric (in the sense of vector superspaces).

We claim that $b^\sigma := b$ (for $\sigma = \pm$) is a metric of $\cA^{\text{(III)}}_{2n}$. Note that this satisfies the required rule of signs $b^+(x^+,y^-) = \eta_{x^+,y^-} b^-(y^-,x^+)$. To prove the claim, we need to prove that $b$ is a non-degenerate even superinvariant supersymmetric bilinear form. It is clear that $b$ is non-degenerate and even. For $x,y,z,w \in T$, we have
\begin{align*}
b(D_{x,y}z,w)&= b\big( b(x,y)z + b(y,z)x + b(x,z)y, w \big) \\
&= -b(y,x)b(z,w) - b(z,y)b(x,w) - b(z,x)b(y,w) \\
&= -b\big( z, b(y,x)w + b(x,w)y + b(y,w)x \big)
= \eta_{x,y,z} b(z, D_{y,x}w),
\end{align*}
so $b$ is superinvariant. Also,
\begin{align*}
b(D_{x,y}z ,w)&= b\big( b(x,y)z + b(y,z)x + b(x,z)y, w \big) \\
&= b(x,y)b(z,w) + b(z,y)b(w,x) + b(z,x)b(w,y) \\
&= b(b(z,w)x + b(w,x)z + b(z,x)w,y)
= \eta_{D_{x,y},D_{z,w}} b(D_{z,w}x,y),
\end{align*}
and since $b$ is superinvariant it follows that $b$ is supersymmetric. We have proven the claim.

Note that we can fix a hyperbolic basis $\{x_i\}_{i=1}^{2n}$ of $T$, i.e., such that $b(x_i,x_j) = \Omega_{ij}$, where
\begin{align} \label{omegaMatrix}
\Omega = (\Omega_{ij}) := \left( \begin{array}{c|c} 0 & \mathrm{I}_n \\
\hline -\mathrm{I}_n & 0\end{array} \right).
\end{align}
Then the triple products of $\cA^{\text{(III)}}_{2n}$ are given by
\begin{align}
\{ x_i, x_j, x_k \}^\sigma = \Omega_{ij} x_k + \Omega_{jk} x_i + \Omega_{ik} x_j.
\end{align}

\item[2)]
The triple products of the antipair $(\cV^{\text{(I)}}_{1,2n}, t)  \otimes (\cA^{\text{(III)}}_{2n}, -b)$ are
\begin{align*}
\{ e_r & \otimes x_i , e_s \otimes x_j, e_t \otimes x_k \}^\sigma = \\
&= \eta_{ x_i, e_s } \Big( \{ e_r, e_s, e_t \}^\sigma \otimes \langle x_i, x_j \rangle x_k
+ \eta_{e_t, x_i} \eta_{e_t,x_j} \langle e_r, e_s \rangle e_t \otimes \{ x_i, x_j, x_k \}^\sigma \Big) \\
&= (\delta_{rs} e_t + \delta_{ts} e_r) \otimes (-\Omega_{ij}) x_k
+ \delta_{rs} e_t \otimes (\Omega_{ij} x_k + \Omega_{jk} x_i + \Omega_{ik} x_j)\\
&= \boxed{  \Omega_{jk} \delta_{rs} e_t \otimes x_i
+ \Omega_{ik} \delta_{rs} e_t \otimes x_j
- \Omega_{ij} \delta_{st} e_r \otimes x_k  },
\end{align*}
and the metric is
$$ \langle e_r \otimes x_i , e_s \otimes x_j \rangle^\sigma
= \langle e_r,e_s \rangle^\sigma \langle x_i,x_j \rangle^\sigma
= \boxed{ -\delta_{rs}\Omega_{ij} }. $$
$\bullet$ Let $\alpha = (0, \bar 1) \in \FF \times \ZZ_2$. Then the tensor-shift (a parity-flip) $\cV := \Big( (\cV^{\text{(I)}}_{1,2n}, t)  \otimes (\cA^{\text{(III)}}_{2n}, -b) \Big)^{[\alpha]}$ is a pair (even superpair) with metric and triple products given by
\begin{align}
\langle e_r \otimes x_i , e_s \otimes x_j \rangle^\sigma
&= \boxed{ \sigma \delta_{rs}\Omega_{ij} }, \\
\{ e_r \otimes x_i , e_s \otimes x_j, e_t \otimes x_k \}^{\sigma}
&= \boxed{
\sigma \big( \Omega_{kj} \delta_{rs} e_t \otimes x_i
+ \Omega_{ki} \delta_{rs} e_t \otimes x_j
+ \Omega_{ij} \delta_{st} e_r \otimes x_k \big)
}.
\end{align}

\item[3)]
Recall that the standard symplectic involution of $\cM_{2n}(\FF)$ is given by
\begin{align}
x = \left( \begin{array}{c|c} A & B \\
\hline C & D \end{array} \right) \longmapsto
\overline{x} := \Omega x^t \Omega^{-1} = \left( \begin{array}{c|c} D^t & -B^t \\
\hline -C^t & A^t \end{array} \right),
\end{align}
for each $x \in \cM_{2n}(\FF)$, where $A,B,C,D \in \cM_n(\FF)$, where $\Omega$ is defined as in \eqref{omegaMatrix}, and where $x^t$ denotes the transpose of $x$.
For the canonical basis $\{ E_{ij} \}_{ij}$ of $\cM_n(\FF)$, the involution is given by
\begin{align*}
\overline{E_{ij}} = \left\lbrace \begin{array}{l}
E_{j+n,i+n}, \text{if } 1\leq i,j \leq n, \\
E_{j-n,i-n}, \text{if } n < i,j \leq 2n, \\
-E_{j-n,i+n}, \text{if } 1\leq i \leq n < j \leq 2n, \\
-E_{j+n,i-n}, \text{if } 1\leq j \leq n < i \leq 2n.
\end{array} \right.
\end{align*}
This is, $\overline{E_{ij}} = \gamma_{ij} E_{[j+n][i+n]}$, where $\gamma_{ij}:= \Omega_{i[i+n]} \Omega_{j[j+n]}$ and $[i]=i \text{ mod } 2n$. We also have
\begin{align*}
E_{ri} \overline{E_{sj}} E_{tk} = E_{ri} \gamma_{sj} E_{[j+n][s+n]} E_{tk}
= \gamma_{sj} \delta_{i [j+n]} \delta_{[s+n] t} E_{rk}
\end{align*}
Let $\cW = (\cM_{2n}(\FF), \cM_{2n}(\FF))$ be the Kantor pair associated to the structurable algebra $(\cM_{2n}(\FF), \inv)$. The triple products of $\cW$ are given by
\begin{align*}
\{ E_{ri}, E_{sj}, E_{tk} \}
&= (E_{ri} \overline{E_{sj}}) E_{tk} + (E_{tk} \overline{E_{sj}}) E_{ri} - (E_{tk} \overline{E_{ri}}) E_{sj}\\
&= \boxed{
\gamma_{sj} \delta_{i [j+n]} \delta_{[s+n] t} E_{rk}
+ \gamma_{sj} \delta_{k [j+n]} \delta_{[s+n] r} E_{ti}
-\gamma_{ri} \delta_{k [i+n]} \delta_{[r+n] s} E_{tj}
}
\end{align*}
for $i,j,k,r,s,t \in \{ 1,..., 2n \}$.
Denote
\begin{align*}
\phi_i :=
\begin{cases}
-1, \quad \text{if $i \in \{1,\ldots n\}$}, \\
+1, \quad \text{if $i \in \{n+1,\ldots 2n\}$},
\end{cases}
\end{align*}
and note that we have the properties
\begin{align*}
\Omega_{ij} = \delta_{[i+n]j} \phi_j, \qquad
\gamma_{ik} \gamma_{jk} = \gamma_{ij} = \phi_i \phi_j, \qquad
\phi_{[i+n]} = - \phi_i, \qquad \delta_{ij} = \delta_{[i+n][j+n]}.
\end{align*}
Denote $\widetilde{E}_{ij} := \phi_i E_{[i+n]j}$.

Our final goal now is to prove the isomorphism $\cW \cong \cV$, with $\cV$ defined as in 2); unfortunately, we will need a square root of $-1$ in $\FF$ to define the isomorphism. First note that
\begin{align*}
\{ E_{ri}, \widetilde{E}_{sj}, E_{tk} \}^+
&= \phi_s \{ E_{ri}, E_{[s+n]j}, E_{tk} \} \\
&= \phi_s \gamma_{[s+n]j} \delta_{rs} \delta_{k[j+n]} E_{ti}
-\phi_s \gamma_{ri} \delta_{rs} \delta_{k[i+n]} E_{tj}
+\phi_s \gamma_{[s+n]j} \delta_{st} \delta_{i[j+n]} E_{rk} \\
&= (\delta_{[j+n]k} \phi_k) \phi_k \phi_s (-\phi_s) \phi_j \delta_{rs} E_{ti} \\
&\quad - (\delta_{[i+n]k} \phi_k) \phi_k \phi_i (\phi_r \phi_s \delta_{rs}) E_{tj} \\
&\quad + (\delta_{[i+n]j} \phi_j) (\phi_s \phi_{[s+n]}) \delta_{st} E_{rk} \\
&= \Omega_{jk} (-\phi_j\phi_k) \delta_{rs} E_{ti}
- (\Omega_{ik} \phi_i \phi_k)  \delta_{rs} E_{tj}
- \Omega_{ij} \delta_{st} E_{rk} \\
&= \boxed{
\Omega_{jk} \delta_{rs} E_{ti}
+ \Omega_{ik} \delta_{rs} E_{tj}
+ \Omega_{ji} \delta_{st} E_{rk}
}
\end{align*}
\begin{align*}
\{ \widetilde{E}_{ri}, E_{sj}, \widetilde{E}_{tk} \}^-
&= \phi_r \phi_t \{ E_{[r+n]i}, E_{sj}, E_{[t+n]k} \} \\
&= \phi_r \phi_t \gamma_{sj} \delta_{k[j+n]} \delta_{rs} E_{[t+n]i}
- \phi_r \phi_t \gamma_{[r+n]i} \delta_{k[i+n]} \delta_{rs} E_{[t+n]j}
+ \phi_r \phi_t \gamma_{sj} \delta_{i[j+n]} \delta_{st} E_{[r+n]k} \\
&= (\delta_{[k+n]j} \phi_j) (\phi_r\phi_s\delta_{rs}) (\phi_t E_{[t+n]i}) \\
&\quad - (\delta_{[k+n]i} \phi_i) \phi_r (-\phi_{r}) \delta_{rs} (\phi_t E_{[t+n]j}) \\
&\quad + (\delta_{[i+n]j} \phi_j) (\phi_s \phi_t \delta_{st}) (\phi_r E_{[r+n]k}) \\
&= \boxed{
\Omega_{kj} \delta_{rs} \widetilde{E}_{ti}
+ \Omega_{ki} \delta_{rs} \widetilde{E}_{tj}
+ \Omega_{ij} \delta_{st} \widetilde{E}_{rk}
}
\end{align*}
Denote $\cE_{ij}^+ := E_{ij}$, $\cE_{ij}^- := \widetilde{E}_{ij}$. Then the triple products are given by
\begin{align}
\{ \cE^\sigma_{ri}, \cE^{-\sigma}_{sj}, \cE^\sigma_{tk} \}^\sigma
= \boxed{
\sigma \big( \Omega_{jk} \delta_{rs} \cE^\sigma_{ti}
+ \Omega_{ik} \delta_{rs} \cE^\sigma_{tj}
+ \Omega_{ji} \delta_{st} \cE^\sigma_{rk} \big)
}.
\end{align}

$\blacksquare$ Finally, assume that there is an element $\bi \in \FF$ such that $\bi^2 = -1$. From the computations above, it is easy to see that we have an isomorphism of metric pairs
\begin{align}
\varphi \colon (\cV, \langle\cdot,\cdot\rangle) \longrightarrow (\cW, \langle\cdot,\cdot\rangle), \qquad
\varphi^\sigma(e_i \otimes x_j) := \bi \cE^\sigma_{ij},
\end{align}
where the metric in $\cW$ is given by
\begin{align}
\langle \cE^\sigma_{ri}, \cE^{-\sigma}_{sj} \rangle^\sigma := -\sigma \delta_{rs} \Omega_{ij}.
\end{align}
We have proven the decomposition:
\begin{align}
(\cW, \langle\cdot,\cdot\rangle) \cong 
\big( (\cV^{\text{(I)}}_{1,2n}, t) \otimes (\cA^{\text{(III)}}_{2n}, -b) \big)^{[(0,\bar 1)]}.
\end{align}

From the isomorphisms above (and by the properties of tensor-shifts) we also get the following decomposition of $(\cW, \langle\cdot,\cdot\rangle)$ using two antipairs as factors:
\begin{align}
(\cW, \langle\cdot,\cdot\rangle) \cong \big( (\cA^{\text{(I)}}_{1,2n}, \widetilde{t})
\otimes (\cA^{\text{(III)}}_{2n}, -b) \big)^{[(2,\bar 0)]}.
\end{align}
Indeed,
\begin{align*}
(\cW, \langle\cdot,\cdot\rangle) &\cong (\cV, \langle\cdot,\cdot\rangle) =
\big( (\cV^{\text{(I)}}_{1,2n}, t) \otimes (\cA^{\text{(III)}}_{2n}, -b) \big)^{[(0,\bar 1)]}
\cong \big( (\cA^{\text{(I)}}_{1,2n}, \widetilde{t})^{[(2,\bar 1)]}
\otimes (\cA^{\text{(III)}}_{2n}, -b) \big)^{[(0,\bar 1)]} \\
&\cong \big( (\cA^{\text{(I)}}_{1,2n}, \widetilde{t})
\otimes (\cA^{\text{(III)}}_{2n}, -b) \big)^{[(2,\bar 1) + (0,\bar 1)]}
\cong \big( (\cA^{\text{(I)}}_{1,2n}, \widetilde{t})
\otimes (\cA^{\text{(III)}}_{2n}, -b) \big)^{[(2,\bar 0)]}.
\end{align*}
Of course, both decompositions are related to the decomposition of $\bAut(\cW, \langle\cdot,\cdot\rangle)$ as a central product of the automorphism group schemes of the factors of $(\cW, \langle\cdot,\cdot\rangle)$.
\end{itemize}
\end{examples}

\begin{remark}
For the Kantor pair in the unitary case, i.e., $\cV = \cV_\cA$ with $\cA = \cM_n(\FF) \oplus \cM_n(\FF)^{op}$ and the exchange involution, it can be checked that a metric is given by the map $t \colon \cA \times \cA \to \FF$, where
\begin{equation}
t(x,y) := t(x_1y_2) + t(x_2y_1).
\end{equation}
Also, Lemma~\ref{lemma.AutKantor} shows that $\bAut(\cV,t) = \bAut(\cV)$. However, in this case, a decomposition as a tensor product of metric generalized Jordan pairs has not been found by the authors, and it might not exist.
\end{remark}

		

	\end{document}